\documentclass[10pt,a4paper]{article}
\usepackage[english]{babel}
\usepackage{polski}
\usepackage{amsmath, amsthm, amsfonts}
\usepackage{tikz-cd}
\usepackage{amssymb}
\usepackage{float}
\usepackage{stackengine}
\usepackage{scalerel}

\usepackage{amssymb}
\usepackage{amsmath}
\usepackage{amsthm}
\usepackage{mathrsfs}  

\usepackage[english]{babel}
\usepackage[utf8]{inputenc}
\usepackage[T1]{fontenc}
\usepackage{todonotes}
\usepackage{xcolor}
\usepackage{tikz-cd}
\usetikzlibrary{calc}

\usepackage{hyperref}
\usepackage{enumitem}

\DeclareMathOperator{\Pic}{Pic}
\DeclareMathOperator{\Proj}{Proj}
\DeclareMathOperator{\Spec}{Spec}
\DeclareMathOperator{\Hilb}{Hilb}
\DeclareMathOperator{\Slip}{Slip}
\DeclareMathOperator{\Ext}{Ext}
\DeclareMathOperator{\Hom}{Hom}

\DeclareMathOperator{\Ann}{Ann}
\DeclareMathOperator{\Eff}{Eff}
\DeclareMathOperator{\Nef}{Nef}
\DeclareMathOperator{\sat}{sat}
\DeclareMathOperator{\Sat}{Sat}
\newcommand{\PP}{\mathbb{P}}%
\newcommand{\ZZ}{\mathbb{Z}}%
\newcommand{\bfa}{\mathbf{a}}%
\newcommand{\NN}{\mathbb{N}}%
\newcommand{\CC}{\mathbb{C}}%
\newcommand{\RR}{\mathbb{R}}%
\newcommand{\Satbar}{\overline{\Sat}}

\newtheorem{theorem}{Theorem}[section]
\newtheorem{corollary}[theorem]{Corollary}
\newtheorem{lemma}[theorem]{Lemma}

\newtheorem{proposition}[theorem]{Proposition}
\theoremstyle{definition}
\newtheorem{definition}[theorem]{Definition}
\newtheorem{remark}[theorem]{Remark}
\newtheorem{example}[theorem]{Example}

\begin{document}

\title{Limits of ideals of points in a toric variety}
\author{Tomasz Ma{\'n}dziuk}
\date{\today}
\maketitle

\begin{abstract} We consider the multigraded Hilbert scheme corresponding to the Hilbert function of a finite number of points in general position in a smooth projective complex toric variety. We develop several criteria for a point of that parameter space to be in the distinguished irreducible component. To obtain the criteria we study the behaviour of the locus of saturated ideals under morphisms of multigraded Hilbert schemes. We apply our results to classify the irreducible multigraded Hilbert schemes corresponding to points in general position in a product of projective spaces.
\end{abstract}

\section{Introduction}
The paper is motivated by the study of secant varieties. Let $X\subseteq \PP^N$ be a variety. We say that $[F]\in \PP^N$ has $X$-rank at most $r$ if $[F]\in \langle x_1, \ldots, x_r\rangle$ for an $r$-tuple $(x_1, \ldots, x_r)\in X^r$. Cases of classical interest include the case where $X$ is the Veronese variety---in which case we obtain the notion of the Waring rank---and the case where $X$ is the Segre variety in which case we get the tensor rank.

A natural approach to studying the set of points of $X$-rank at most $r$ is to consider its Zariski closure. In that way one obtains a projective variety $\sigma_r(X)$ called the $r$-th secant variety. This allows the use of methods from algebraic geometry. Among other properties of these varieties people study their dimensions, see for instance \cite{AH95},\cite{BBC12}, \cite{AB13}, \cite{GO22} and their equations \cite{Rai10}, \cite{LO13}, \cite{BB14}. See also \cite{IK06} and \cite{Lan12}.
We say that $[F]$ has $X$-border rank at most $r$ if it is in $\sigma_r(X)$. The closure involved in the definition of secant varieties is not easy to understand and there are not many methods for studying the $X$-border rank. For example the equations of secant varieties are not known in general and many known classes of equations of secant varieties are in fact equations of larger varieties---the cactus varieties \cite{Gal16}.

The border apolarity lemma \cite[Thm.~3.15]{BB19} is a recent tool for studying the $X$-border rank in the case where $X$ is a smooth complex projective toric variety. This includes in particular the classical cases of Segre-Veronese varieties.
The result establishes a connection between the set of all points in $\PP^N$ whose $X$-border rank is at most $r$ and the set of all points in a certain irreducible component $\Slip_{r, X}$ of a multigraded Hilbert scheme $\Hilb_{S[X]}^{h_{r, X}}$ (we recall the notion of a multigraded Hilbert scheme and explain the notation in Subsection~\ref{subsec:notation}). This shifts the problem from understanding the closure in $\PP ^N$ to understanding the closure in $\Hilb_{S[X]}^{h_{r, X}}$ and allows for using a different set of tools.
 Motivated by this result we develop some necessary criteria for a point of that parameter space to be in $\Slip_{r,X}$. We also use these criteria to classify all reducible multigraded Hilbert schemes $\Hilb_{S[X]}^{h_{r, X}}$ in the case that $X$ is a product of projective spaces. This is a natural continuation of papers \cite{Man20}~and~\cite{JM22} in which we presented some necessary and one sufficient condition for $[I]$ to be in $\Slip_{r, X}$. The second paper considers a more general problem of understanding which ideals are limits of saturated ideals. We present some of our results in this generality. The border apolarity lemma has been successfully applied to the study of border ranks of tensors. See for example \cite{CHL19}, \cite{HMV20},  \cite{GMR20} and \cite{Fla22}. We expect that the criteria developed in this paper will allow to obtain new results on border ranks.
Our main motivation is the study of the case of $X$ being the product of projective spaces. However, there are some results concerning the secant varieties and related objects for more general toric varieties \cite{CS07}, \cite{GRV18} so we present our results in a greater generality.

\subsection{Notation}\label{subsec:notation}

Let $X$ be a smooth projective complex toric variety. It has a Cox ring $S[X]$ that is a polynomial ring with variables corresponding to torus invariant Cartier divisors on $X$. This ring is graded by the Picard group $\Pic(X)$ of $X$. It has a distinguished $\Pic(X)$-homogeneous ideal $B(X)$ called the irrelevant ideal. See \cite{CLS11} for the precise definitions that we use. Given a $\Pic(X)$-graded $S[X]$-module $M$ we denote by $M_{[D]}$ its vector subspace of all homogeneous elements of degree $[D]$. By $H_M\colon \Pic(X)\to \NN\cup \{\infty\}$ we denote the Hilbert function of $M$, i.e. we have $H_M([D]) = \dim_\CC M_{[D]}$.

Given a positive integer $r$ we consider the Hilbert function $h_{r,X}\colon \Pic(X)\to \mathbb{N}$ given by $h_{r,X}([D]) = \min\{\dim_\CC S[X]_{[D]}, r\}$. This is the Hilbert function of the quotient algebra of the ideal of $r$ points in general position in $X$.

By $\Hilb_{S[X]}^{h_{r,X}}$ we denote the multigraded Hilbert scheme parametrizing all $\Pic(X)$-homogeneous ideals $I$ of $S[X]$ such that the Hilbert function of $S[X]/I$ is $h_{r,X}$. This parameter space exists and is a projective scheme \cite[Cor.~1.2]{HS04}. It has a distinguished irreducible component $\Slip_{r,X}$ (see \cite[Prop.~3.13]{BB19}) which is the closure of the set of all ideals that are $B(X)$-saturated and radical.

We denote by $\Eff(X)$ the set of all $[D]\in \Pic(X)$ with $H^0(X, \mathcal{O}_X(D)) \neq 0$ and by $\Nef(X)$ the set of all $[D]\in \Pic(X)$ with $D$ a nef line bundle.

\subsection{Main results}
The main results of this paper are necessary criteria for a point of $\Hilb_{S[X]}^{h_{r,X}}$ to be in $\Slip_{r,X}$. One of them is presented in an abstract version in Proposition~\ref{prop:main_observation} and has two consequences: Theorems~\ref{thm:toric_fibration} and Theorem~\ref{thm:embd_criterion} whose special versions where $X$ is a product of projective spaces we discuss here. Both of these criteria are of the following form. We have $[I]\in \Hilb_{S[X]}^{h_{r,X}}$ and we want to know if it belongs to $\Slip_{r,X}$. We associate to $[I]$ an ideal $[J]\in \Hilb_{S[Y]}^{h_{r,Y}}$ for a different smooth projective complex toric variety $Y$. A necessary condition for $[I]$ to be in $\Slip_{r, X}$ is that $[J]\in \Slip_{r,Y}$. In applications that we present, $Y$ is a simpler variety than $X$ so we may use what we know about $\Slip_{r,Y}$ to obtain insight into $\Slip_{r,X}$

Our first result concerns projections from the products of projective spaces. We start with an example.

\begin{example}\label{ex:projections}
Let $a,b,c$ be positive integers and $X=\PP^a\times \PP^b\times \PP^c$.
Its Cox ring is the $\ZZ^3$-graded polynomial ring $\CC[\alpha_0,\ldots, \alpha_a, \beta_0,\ldots, \beta_b, \gamma_0, \ldots, \gamma_c]$ with $\deg(\alpha_i) = (1,0,0)$,
$\deg(\beta_j) = (0,1,0)$ and $\deg(\gamma_k)= (0,0,1)$ for every $0\leq i \leq a$, $0\leq j \leq b$ and $0\leq k \leq c$. Let $Y = \PP^a$ (respectively $\PP^a\times \PP^b$). Its Cox ring is the subring $S[Y] = \CC[\alpha_0, \ldots, \alpha_a]$ (respectively, $S[Y] = \CC[\alpha_0, \ldots, \alpha_a, \beta_0, \ldots, \beta_b]$) of $S[X]$. 
Given $[I]\in \Hilb_{S[X]}^{h_{r,X}}$ the ideal $J = I\cap S[Y]$ determines a point $[J]\in \Hilb_{S[Y]}^{h_{r,Y}}$. Our result shows that if $[I] \in \Slip_{r,X}$ then $[J]\in \Slip_{r,Y}$.

While our result works for any positive integers $a,b,c$ and $r$ for illustration we choose specific values: $a=b=c=5$ and $r=7$. 
If $[I]\in \Slip_{r,X}$ then $I$ has at least $129$ minimal generators:
\begin{itemize}
    \item $21-7=14$ generators of each of the degrees $(2,0,0)$, $(0,2,0)$ and $(0,0,2)$
    \item $36-7=29$ minimal generators of each of the degrees $(1,1,0), (1,0,1)$ and $(0,1,1)$.
\end{itemize}
On the other hand an ideal defining a $7$-tuple of points in $\PP^5$ in general position can be generated by $14$ elements of degree $2$. In general, the number of minimal generators of $I\cap S[Y]$ is significantly smaller than the number of minimal generators of $I$. It follows that $\Slip_{r,Y}$ should be easier to understand than $\Slip_{r,X}$. 
\end{example}

The general version for projections from products of projective spaces is as follows. A version for more general toric varieties is presented in Theorem~\ref{thm:toric_fibration} and includes the claim that the map of $\Slip$'s is surjective.
\begin{theorem}\label{thm:fibration_intro}
Let $Y=\PP^{n_1}\times \cdots \times \PP^{n_d}$ and $X = Y  \times \PP^{m_1}\times \cdots \times \PP^{m_e}$.
Let $S[Y]\subseteq S[X]$ be the Cox rings of $Y$ and $X$.
Let $r$ be a positive integer and $[I]\in \Hilb_{S[X]}^{h_{r,X}}$. We have $[I\cap S[Y]] \in \Hilb_{S[Y]}^{h_{r,Y}}$ and if $[I]\in \Slip_{r,X}$, then $[I\cap S[Y]]\in \Slip_{r,Y}$.
\end{theorem}

The second result (Theorem~\ref{thm:embd_intro}) is motivated by the Segre embedding.
\begin{example}\label{ex:segre}
    Let $X=\PP^{a}\times \PP^{b}\times \PP^c$ and $N=(a+1)(b+1)(c+1)-1$. Assume that  $N +1 \geq  r$. Let 
    \[
    T = \CC[t_{ijk}|0\leq i \leq a, 0\leq j \leq b, 0\leq k \leq c]
    \]
    and consider the natural map $\varphi\colon T\to S$ given by $t_{ijk}\mapsto \alpha_i\beta_j\gamma_k$.  Given a point $[I]\in \Hilb_{S[X]}^{h_{r,X}}$, the ideal $J=\varphi^{-1}(I) \subseteq T$
    defines a point $[J]\in \Hilb_{S[\PP^N]}^{h_{r, \PP^N}}$. We show that if $[I]\in \Slip_{r,X}$ then $[J]\in \Slip_{r,\PP^N}$.

    As in Example~\ref{ex:projections} for illustration of the result we choose $a=b=c=5$ and $r=7$. An ideal $[I]\in \Hilb_{S[X]}^{h_{r,X}}$ has at least $129$ minimal generators and an ideal defining a $7$-tuple of points in general position in $\PP^{215}$ has $230$ minimal generators:
    \begin{itemize}
        \item $216-7 = 209$ linear generators
        \item $21$ quadratic generators.
    \end{itemize}
    The total number of generators is large but
        in order to show that $[J]\notin \Slip_{7, \PP^{215}}$ we may change the coordinates so that the $209$ linear generators are variables and then consider the image $J'$ of $J$ in the ring $S[\PP^{215}]/(J_1)$. It follows from \cite[Prop.~3.1]{CEVV09}  that if $[J']\notin \Slip_{7, \PP^6}$ then $[J]\notin \Slip_{7, \PP^{215}}$. 
\end{example}

Our result in the case of products of projective spaces is the following. See Theorem~\ref{thm:embd_criterion} for a version for more general toric varieties instead of $X$.
\begin{theorem}\label{thm:embd_intro}
Let $X=\PP^{n_1}\times \cdots \times \PP^{n_d}$ and $\mathbf{u}\in \ZZ_{\geq 0}^{d}\setminus \{\mathbf{0}\}$. Let $(g_0, \ldots, g_k)$  be a basis of $S[X]_\mathbf{u}$ and let $\varphi\colon \CC[t_0,\ldots, t_k] \to S[X]$ be given by $t_i\mapsto g_i$. If $k+1\geq r$ then the map $[I]\mapsto [\varphi^{-1}(I)]$ defines a morphism $\pi\colon \Hilb_{S[X]}^{h_{r,X}} \to \Hilb_{S[\PP^k]}^{h_{r,\PP^k}}$ that maps $\Slip_{r,X}$ into $\Slip_{r,\PP^k}$.
\end{theorem}

We also obtain a criterion based on the dimension of the tangent space (Proposition \ref{prop:cg_subsets_give_necessary_conditions}).
We mainly use its consequence Corollary~\ref{cor:ts_products} which concerns the case of products of projective spaces. However, note that even for this variety we may replace the ideals $\mathfrak{a}_i$ in the statement below by other ideals as described in Proposition~\ref{prop:ts_products_general} or even more generally Proposition~\ref{prop:cg_subsets_give_necessary_conditions}.

\begin{proposition}[Corollary~\ref{cor:ts_products}]\label{prop:ts_intro}
Let $X=\PP^{n_1}\times \cdots \times \PP^{n_d}$ for some $d\geq 2$ and $n_1,\ldots, n_d \geq 1$.
Let $S[\PP^{n_i}] = \CC[\alpha_{i0}, \ldots, \alpha_{in_i}]$ be the Cox ring of the $i$-th factor.
Let $\mathfrak{a}_i$ be the extension of the irrelevant ideal $(\alpha_{i0}, \ldots, \alpha_{in_i})$ of $\PP^{n_i}$ to the Cox ring of $X$.
If $[I] \in \Slip_{r,X}$, then for every $i\in \{1,2,\ldots, d\}$
\[
\dim_\CC \Hom_S(I+\mathfrak{a}_i^2, S/(I+\mathfrak{a}_i^2))_\mathbf{0} \geq r\cdot \dim X.
\]
\end{proposition}

We illustrate our criteria with an example.
\begin{example}\label{ex:explicit}
Let $X=\PP^3\times \PP^3\times \PP^3$ and $S[X] = \CC[\alpha_0,\ldots, \alpha_3, \beta_0, \ldots, \beta_3, \gamma_0, \ldots, \gamma_3]$ be its Cox ring.
Consider the ideal

\begin{align*}
I & = (\alpha_0,\alpha_1,\alpha_2)^2 + (\beta_2,\beta_3)(\beta_0,\beta_1,\beta_2) + (\beta_1^3) + (\gamma_2)(\gamma_0,\ldots, \gamma_3) + (\gamma_1^2,\gamma_1\gamma_3, \gamma_0^2\gamma_3) \\
& + (\alpha_0,\ldots, \alpha_3)(\beta_2,\beta_3)+(\alpha_1,\alpha_2)(\beta_0,\beta_1) + (\alpha_0\beta_1^2,\alpha_3\beta_1^2) + (\gamma_0,\gamma_1,\gamma_2)(\alpha_0,\ldots, \alpha_3, \beta_0,\ldots, \beta_3). 
\end{align*}

We have $[I]\in \Hilb_{S[X]}^{h_{4, X}}$ and we show that $[I]\notin \Slip_{4, X}$. If we apply Proposition~\ref{prop:ts_intro} with $i\in\{1,2,3\}$ the dimensions of the
spaces of homomorphisms are $48, 53$ and $56$ respectively. Therefore, we cannot conclude that $[I]\notin \Slip_{4, X}$.
Let $B$ be the ideal generated by $S[X]_{(1,1,1)}$ and $J = (I\colon B^\infty)$ be the saturation. We have $\dim_\CC \Ext_S^1(J/I, S/J)_\mathbf{0} = 3$ so the criterion from \cite[Thm.~3.4]{JM22} (with our choice of $J$) is also not sufficient to conclude that $[I]\notin \Slip_{4,X}$. Similarly, if we use Theorem~\ref{thm:fibration_intro} for the three projections onto $\PP^3$'s we obtain ideals
\begin{align*}
    I' &= (\alpha_2^2,\alpha_1\alpha_2,\alpha_0\alpha_2,\alpha_1^2,\alpha_0\alpha_1,\alpha_0^2) \subseteq \CC[\alpha_0, \alpha_1, \alpha_2, \alpha_3] \\
    I'' &= (\beta_2\beta_3,\beta_1\beta_3,\beta_0\beta_3,\beta_2^2,\beta_1\beta_2,\beta_0\beta_2,\beta_1^3)\subseteq \CC[\beta_0, \beta_1, \beta_2, \beta_3]\\
    I''' & = (\gamma_2\gamma_3,\gamma_1\gamma_3,\gamma_2^2,\gamma_1\gamma_2,\gamma_0\gamma_2,\gamma_1^2,\gamma_0^2\gamma_3) \subseteq \CC[\gamma_0, \gamma_1, \gamma_2, \gamma_3].
\end{align*}
It follows from \cite[Prop.~4.3]{JM22} that $[I'], [I''], [I''']$ are all points of $\Slip_{4, \PP^3}$ so again we cannot conclude that $[I]\notin \Slip_{4, X}$.
Consider the projection map $X\to \PP^3\times \PP^3$ onto the first two factors. If $[I]\in \Slip_{4, X}$, then by Theorem~\ref{thm:fibration_intro} we have $[J]\in \Slip_{4, \PP^3\times \PP^3}$ where 
\begin{align*}
J = &(\beta_2\beta_3,\beta_1\beta_3,\beta_0\beta_3,\beta_2^2,\beta_1\beta_2,\beta_0\beta_2,\alpha_3\beta_3,\alpha_3\beta_2,\alpha_2\beta_
      3,\alpha_2\beta_2,\alpha_2\beta_1,\alpha_2\beta_0,\alpha_1\beta_3,\alpha_1\beta_2,\alpha_1\beta_1,\alpha_1\beta_0,\\
      & \alpha_0\beta_3,\alpha_0\beta_2,
      \alpha_2^2,\alpha_1\alpha_2,\alpha_0\alpha_2,\alpha_1^2,\alpha_0\alpha_1,\alpha_0^2,\beta_1^3,\alpha_3\beta_1^2,\alpha_0\beta_1^2)\subseteq T :=  \CC[\alpha_0,\ldots, \alpha_3, \beta_0,\ldots, \beta_3].
\end{align*}
Let $\mathfrak{a} = (\alpha_0, \ldots, \alpha_3)$. We compute that $\dim_\CC \Hom_T(J+\mathfrak{a}^2, T/(J+\mathfrak{a}^2))_\mathbf{0} = 23$ which by Proposition~\ref{prop:ts_intro} shows that $[J]\notin \Slip_{4, \PP^3\times \PP^3}$ and hence $[I]\notin \Slip_{4, X}$.
\end{example}

The developed criteria are applied to classify the reducible cases of $\Hilb_{S[X]}^{h_{r,X}}$ where $X$ is the product of projective spaces.

\begin{theorem}[Theorem~\ref{thm:classification}]
Let $r, d, n_1, \ldots, n_d$ be positive integers and $X=\PP^{n_1}\times \cdots \times \PP^{n_d}$.
The scheme $\Hilb_{S[X]}^{h_{r, X}}$ is irreducible if and only if one of the following holds:
\begin{enumerate}
\item $r=1$;
\item $d=1$ and $n_1 = 1$;
\item $d=1$ and $r\leq 3$.
\end{enumerate}
\end{theorem}

\subsection{Strengths of the criteria}
Our criteria have three strengths. For simplicity of their discussion we assume that $X=\PP^{n_1}\times \cdots \times \PP^{n_d}$ is embedded into $\PP^N$ by the Segre embedding. While the third benefit might loose its importance when more tools are developed for studying $\Slip_{r,X}$, it seems that the first two will remain significant. 

\subsubsection{Lowered time of computations}
As illustrated in Examples~\ref{ex:projections},~\ref{ex:segre}~and~\ref{ex:explicit} the results from Theorems~\ref{thm:fibration_intro} and \ref{thm:embd_intro} are of the form that we replace a more complicated ideal $I$ by a less complicated ideal $J$.
    If we are able to show that $[J]$ is not in the corresponding $\Slip$ component then neither is $[I]$. Suppose that we have some criteria for being in $\Slip$ that are easy to implement in a computer algebra system but the computation time of their verification is dependent on the complexity of the ideal. One example of such a criterion is \cite[Thm.~3.4]{JM22}. Another example is Proposition~\ref{prop:ts_intro}. Assume that we have a set of ideals in $\Hilb_{S[X]}^{h_{r,X}}$ all of which are contained in $\Ann(T)$ for some tensor $T \in \CC^{N+1}$ and we want to show that none of these ideals is in $\Slip_{r,X}$.  Instead of verifying the criteria for the ideals in $S[X]$ we may first verify these criteria for the corresponding ideals in the Cox rings of lower dimensional varieties. This should lead to significantly faster computations and could shorten the list of ideals that we want to consider more deeply.
\subsubsection{Restrictions on apolar ideals worth considering}
If we want to show that a tensor $T\in \CC^{N+1}$ has border rank greater than $r$ then, by border apolarity \cite[Thm.~3.15]{BB19}, we should construct all apolar ideals contained in $\Ann(T)$ that have the correct Hilbert function and verify that none of them is in $\Slip_{r,X}$. Typically there are many such ideals. A fundamental result of the border apolarity theory is \cite[Thm~4.3]{BB19} which allows to limit the set of ideals under consideration by taking advantage of the symmetries of $T$. Criteria from Theorems~\ref{thm:fibration_intro}~and~\ref{thm:embd_intro} as well as Proposition~\ref{prop:ts_intro} provide another way of limiting the set of candidate ideals. In each case, we replace an ideal $I$ by a different ideal $J$ which in fact depends only on some degrees of $I$. Therefore, we are sometimes able to show that $[I]\notin \Slip_{r,X}$ even without constructing $I$ in all degrees.
\subsubsection{The case of $\PP^n$ is better understood}
The final benefit of the criteria from Theorems~\ref{thm:fibration_intro}~and~\ref{thm:embd_intro} is that we can reduce showing that $[I]\notin \Slip_{r,X}$ to showing that $[J]\notin \Slip_{r, \PP^n}$. The case of $\Slip_{r,\PP^n}$ is for the time being better understood than the general case.

%

\subsubsection*{Acknowledgments}
The initial versions of some of the results in this paper come from my PhD thesis which was supervised by Jarosław Buczyński and Joachim Jelisiejew. I am grateful to them for 
introducing me to the subject and many conversations and explanations.
I would like to thank Alessandra Bernardi, Jarosław Buczyński, Joachim Jelisiejew and Filip Rupniewski for providing valuable suggestions on how to improve the presentation.
I would like to thank the National Science Center for supporting my research (project number 2019/33/N/ST1/00858). 
Computations were done in Macaulay2 \cite{M2}.

\section{A natural map of multigraded Hilbert schemes}

We state our main observation of this section in a more abstract setting than we need to highlight its generality. The two main cases of interest are described in Subsections~\ref{sub:fibration} and \ref{sub:embedding}. 

Let $\Bbbk$ be a field. Assume that $S$ is a polynomial algebra over $\Bbbk$ graded by an abelian group $A$ with a chosen homogeneous ideal $B$. Let $H\colon A\to \mathbb{N}$ be a Hilbert function. We consider the multigraded Hilbert scheme $\Hilb_S^H$ which parametrizes all $A$-homogeneous ideals in $S$ such that the quotient algebra has Hilbert function $H$. We denote by $\Satbar^H$ the closure of the locus of those ideals in  $\Hilb_S^H$ that are $B$-saturated.
By $\Slip^H$ we denote the closure of the locus of those radical ideals in $\Hilb_S^H$ that are $B$-saturated.
Note that it may happen that these sets are empty.

\begin{proposition}\label{prop:main_observation}
Assume that $S$ and $T$ are polynomial rings over a field $\Bbbk$ graded by abelian groups $A_S$ and $A_T$ respectively. Assume that $B_S$ and $B_T$ are distinguished homogeneous ideals in $S$ and $T$, respectively.
Let $\phi\colon A_S \to A_T$ be a group homomorphism and let $\varphi\colon S\to T$ be a graded $\Bbbk$-algebra homomorphism with respect to $\phi$, i.e. $\varphi(S_\bfa)\subseteq T_{\phi(\bfa)}$ for every $\bfa\in A_S$. The following hold:
\begin{enumerate}[label = (\alph*)]
\item\label{it:prop_p1} if for every $\bfa\in A_S$, $\varphi_\bfa\colon S_\bfa \to T_{\phi(\bfa)}$ is surjective then for every Hilbert function $H\colon A_T\to \mathbb{N}$ there is a natural morphism $\pi\colon \Hilb^H_T \to \Hilb_S^{H\circ \phi}$ given  by $[I]\mapsto [\varphi^{-1}(I)]$;
\item\label{it:prop_p2}  if furthermore $S$ and $T$ are positively graded and $B_T \subseteq \sqrt{\varphi(B_S)\cdot T}$, then $\pi$ maps every $B_T$-saturated ideal to a $B_S$-saturated ideal. Thus, it maps
$\Satbar^H$ into $\Satbar^{H\circ \phi}$ and $\Slip^H$ into $\Slip^{H\circ \phi}$.
\end{enumerate}
\end{proposition}
\begin{proof}
Let $R$ be a $\Bbbk$-algebra and let $I_R$ be a $\Spec (R)$-point of $\Hilb_T^H$, i.e. $I_R\subseteq T\otimes_\Bbbk R$ is an $A_T$-homogeneous ideal
 such that for every $\mathbf{b}\in A_T$ the $R$-module $\left(\frac{T\otimes_{\Bbbk} R}{I_R} \right)_\mathbf{b}$ is locally free of rank
 $H(\mathbf{b})$. We show that $(\varphi\otimes_\Bbbk \operatorname{id}_R)^{-1}(I_R)$ is a $\Spec (R)$-point of $\Hilb_S^{H\circ \phi}$.
Pick $\bfa \in A_S$. The map $\varphi_{\bfa} \otimes_\Bbbk \operatorname{id}_R \colon S_\bfa\otimes_\Bbbk R \to T_{\phi(\bfa)}\otimes_\Bbbk R$ is surjective. Therefore, 
\[
\frac{S_\bfa \otimes_\Bbbk R}{((\varphi\otimes_\Bbbk \operatorname{id}_R)^{-1}(I_R))_{\bfa}} = \frac{S_\bfa \otimes_\Bbbk R}{(\varphi_{\bfa} \otimes_\Bbbk \operatorname{id}_R)^{-1}((I_R)_{\phi(\bfa)})} \cong \frac{T_{\phi(\bfa)}\otimes_{\Bbbk} R}{(I_R)_{\phi(\bfa)}}.
\]
Hence it is a locally free $R$-module of rank $H(\phi(\bfa))$. Therefore, $I_R\mapsto (\varphi\otimes_\Bbbk \operatorname{id}_{R})^{-1}(I_R)$
defines a map of sets of $\Spec(R)$-points of these schemes $\Hilb^H_{T}(R) \to \Hilb^{H\circ \phi}_S(R)$. To show that we have defined a natural transformation of functors of points, 
consider a homomorphism $\lambda\colon R\to R'$ of $\Bbbk$-algebras and $I_R\in \Hilb^H_{T}(R)$. 
We need to verify that $\left((\varphi\otimes_\Bbbk \operatorname{id}_R)^{-1}(I_R)\right)\otimes_R R'$ and $(\varphi\otimes_\Bbbk \operatorname{id}_{R'})^{-1}(I_R\otimes_R R')$ 
are the same ideal in $S\otimes_\Bbbk R'$. By the above, we have that for every $\bfa \in A_S$, the degree $\bfa$ parts of both of these ideals are locally free $R'$-modules of the same rank. Furthermore, we have $\left((\varphi\otimes_\Bbbk \operatorname{id}_R)^{-1}(I_R)\right)\otimes_R R' \subseteq (\varphi\otimes_\Bbbk \operatorname{id}_{R'})^{-1}(I_R\otimes_R R')$. Therefore, these ideals coincide. This finishes the proof of part~\ref{it:prop_p1}.

We proceed to the proof of the second part of the proposition. Since $S$ and $T$ are positively graded, the schemes $\Hilb^H_T$ and $\Hilb^{H\circ \phi}_S$ are projective by \cite[Cor.~1.2]{HS04}. In particular, $\pi$ is a closed map. Since the preimage of a radical ideal is a radical ideal it is enough to show that under the assumptions from part~\ref{it:prop_p2} the preimage of a $B_T$-saturated ideal is a $B_S$-saturated ideal.
Since $T$ is a noetherian ring we have $B_T^k \subseteq \varphi(B_S)\cdot T$ for some $k\in \NN$. Let $I$ be a $B_T$-saturated homogeneous ideal of $T$
and assume that $g\in (\varphi^{-1}(I)\colon B_S)$. We have
\[
\varphi(g)\cdot B_T^k \subseteq \varphi(g)\cdot \varphi(B_S)\cdot T = \varphi(g\cdot B_S)\cdot T \subseteq \varphi(\varphi^{-1}(I))\cdot T \subseteq I.
\]
Therefore, $\varphi(g)\in I$ since $I$ is $B_T$-saturated. It follows that $g\in \varphi^{-1}(I)$, so $\varphi^{-1}(I)$ is $B_S$-saturated.
\end{proof}

\subsection{Surjective map with connected fibers}\label{sub:fibration}

In this subsection we restrict our attention to smooth projective complex toric varieties and apply Proposition~\ref{prop:main_observation} in a more geometric setting. We need to recall the quotient construction of a toric variety. This also allows us to fix some notation.

\begin{remark}\label{rmk:quotient_construction}
Let $Z$ be a smooth complex projective toric variety. Let $\overline{Z} = \Spec S[Z]$. Since $S[Z]$ is $\Pic(Z)$-graded, there is a natural action of the corresponding torus $\mathbb{T}_Z = \Spec \CC[\Pic(Z)]$ on $\overline{Z}$. Furthermore, $B(Z)$ is a homogeneous ideal so the action restricts to an action on the open subset $\widehat{Z}=\overline{Z}\setminus V(B(Z))$. The variety $Z$ is the geometric quotient $\pi_Z \colon \widehat{Z} \to Z$ of $\widehat{Z}$ by this action. This is presented in \cite[Thm.~5.1.11]{CLS11}. 
We  denote the inclusion $\widehat{Z} \subseteq \overline{Z}$ by $i_Z$.
\end{remark}

We use the following notion of a lift of a morphism to Cox rings.

\begin{definition}\label{def:lift}
Suppose that $f\colon X\to Y$ is a morphism of smooth projective complex toric varieties and let $\phi\colon \Pic(Y)\to \Pic(X)$ be the pullback map. We call a $\CC$-algebra homomorphism $\varphi \colon S[Y]\to S[X]$ a lift of $f$ if the following conditions are satisfied:
\begin{enumerate}[label=(\roman*)]
\item\label{it:deflift1} $\varphi$ is a graded homomorphism of graded rings with respect to $\phi$, i.e. $\varphi (S[Y]_{[D]})\subseteq S[X]_{\phi([D])}$ for every $[D]\in \Pic(Y)$;
\item\label{it:deflift2} the corresponding morphism of affine spaces $\overline{f} \colon \overline{X} \to \overline{Y}$ restricts to a morphism $\widehat{f} \colon \widehat{X} \to \widehat{Y}$;
\item\label{it:deflift3} $\pi_Y\circ \widehat{f} = f\circ \pi_X$.
\end{enumerate}
\end{definition}

We restate condition \ref{it:deflift2} of Definition~\ref{def:lift} algebraically.

\begin{lemma}\label{lem:characterization_of_property_2}
Let $f\colon X\to Y$ be a morphism of smooth projective complex toric varieties. A homomorphism of $\CC$-algebras  $\varphi\colon S[Y]\to S[X]$ satisfies condition~\ref{it:deflift2} of Definition~\ref{def:lift} if and only if $B(X) \subseteq \sqrt{\varphi(B(Y))\cdot S[X]}$.
\end{lemma}
\begin{proof}
The morphism $\overline{f}\colon \overline{X}\to \overline{Y}$ restricts to a morphism $\widehat{f} \colon \widehat{X}\to \widehat{Y}$ if and only if $\overline{f}^{-1}(V(B(Y))) \subseteq V(B(X))$. Since $\overline{f}^{-1}(V(B(Y))) = V(\varphi(B(Y)))$ we conclude that $\varphi$ satisfies condition~\ref{it:deflift2} of Definition~\ref{def:lift} if and only if $V(\varphi(B(Y)))\subseteq V(B(X))$ which is equivalent to $\sqrt{B(X)}\subseteq \sqrt{\varphi(B(Y))\cdot S[X]}.$
Since the ideal $B(X)$ is a square-free monomial ideal we have $\sqrt{B(X)} = B(X)$ which finishes the proof.
\end{proof}

The next lemma concerns properties of a lift of a surjective morphism with connected fibers.

\begin{lemma}\label{lem:connected_fibers_properties}
Let $f\colon X\to Y$ be a morphism of smooth projective complex toric varieties such that $f_*\mathcal{O}_X \cong \mathcal{O}_Y$. Let $\phi\colon \Pic(Y)\to \Pic(X)$ be the pullback and $\varphi \colon S[Y] \to S[X]$ be a lift of $f$.
\begin{enumerate}[label=(\alph*)]
\item\label{it:con_fib_lem_1} For every $[D]\in \Pic(Y)$, the vector spaces $S[Y]_{[D]}$ and $S[X]_{\phi([D])}$ have the same dimension.
\item\label{it:con_fib_lem_2} The pullback map $\phi\colon \Pic(Y)\to \Pic(X)$ is injective.
\item\label{it:con_fib_lem_3} For every $[D]\in \Pic(Y)$ the homomorphism $\varphi$ induces an isomorphism of the $\CC$-vector spaces $S[Y]_{[D]} \to S[X]_{\phi([D])}$. 
\end{enumerate}
\end{lemma}
\begin{proof}
By \cite[Prop.~5.3.7]{CLS11}, there are isomorphisms 
\[
S[X]_{\phi([D])}\cong H^0(X, f^*\mathcal{O}_Y(D)) \text{ and } [Y]_{[D]}\cong H^0(Y, \mathcal{O}_Y(D)).
\]
Therefore, to prove part~\ref{it:con_fib_lem_1} it is sufficient to show that $H^0(X, f^*\mathcal{O}_Y(D)) = H^0(Y, f_*f^*(\mathcal{O}_Y(D)))$ is isomorphic with $H^0(Y, \mathcal{O}_Y(D))$.
By the projection formula we have 
\[
H^0(Y, f_*(f^*\mathcal{O}_Y(D)))\cong H^0(Y, f_*(\mathcal{O}_X)\otimes \mathcal{O}_Y(D))
\]
which is isomorphic to $H^0(Y, \mathcal{O}_Y(D))$ by the assumption that $f_*\mathcal{O}_X\cong \mathcal{O}_Y$.

Part~\ref{it:con_fib_lem_2} follows from the projection formula and $f_*\mathcal{O}_X\cong \mathcal{O}_Y$.

Since, by part~\ref{it:con_fib_lem_2}, the pullback map $\phi$ is injective the corresponding map of algebraic tori $\mathbb{T}_X\to \mathbb{T}_Y$ is dominant and hence surjective by \cite[Prop.~1.1.1]{CLS11}.
The morphism $f$ is projective and satisfies $f_*\mathcal{O}_X \cong \mathcal{O}_Y$ so it is surjective. We claim that $\overline{f}\colon \overline{X} \to \overline{Y}$ is dominant. It is enough to show that $\widehat{f} \colon \widehat{X}\to \widehat{Y}$ is surjective. Let $\widehat{y} \in \widehat{Y}$. Since $f$ and $\pi_X$ are surjective, there is a point $\widehat{x}\in \widehat{X}$ with $f\circ \pi_X(\widehat{x}) = \pi_Y(\widehat{y})$. It follows that there is an element $t\in \mathbb{T}_Y$ with $t\cdot (\widehat{f}(\widehat{x})) = \widehat{y}$. Using the fact that the map of tori is surjective and that $\widehat{f}$ is equivariant we conclude that there is $t'\in \mathbb{T}_X$ with $\widehat{f}(t'\cdot \widehat{x}) = \widehat{y}$.
We have showed that $\overline{f}$ is dominant so $\varphi$ is injective. Thus, part~\ref{it:con_fib_lem_3} follows from part~\ref{it:con_fib_lem_1}.
\end{proof}

\begin{theorem}\label{thm:toric_fibration}
Let $f\colon X\to Y$ be a morphism between smooth projective complex toric varieties such that $f_*\mathcal{O}_X\cong \mathcal{O}_Y$.
\begin{enumerate}[label=(\alph*)]
\item\label{it:thm_fib_1} There exists $\varphi \colon S[Y]\to S[X]$ such that for every positive integer $r$ the assignment $[I]\mapsto [\varphi^{-1}(I)]$ defines a morphism $\pi\colon \Hilb_{S[X]}^{h_{r,X}} \to \Hilb_{S[Y]}^{h_{r,Y}}$ that maps $\Slip_{r,X}$ into $\Slip_{r,Y}$.
\item\label{it:thm_fib_2} The map $\Slip_{r,X} \to \Slip_{r,Y}$ is surjective.
\end{enumerate}
\end{theorem}

\begin{proof}
Let $\phi\colon \Pic(Y)\to \Pic(X)$ be the pullback map. By \cite[Thm.~3.2]{cox95b} there exists a lift $\varphi \colon S[Y]\to S[X]$
of $f$ as in Definition~\ref{def:lift}. Therefore, part~\ref{it:thm_fib_1} follows from Lemmas~\ref{lem:characterization_of_property_2} and~\ref{lem:connected_fibers_properties}\ref{it:con_fib_lem_3} and Proposition~\ref{prop:main_observation}.

We proceed to the proof of part~\ref{it:thm_fib_2}. For a smooth projective complex toric variety $Z$ and a positive integer $r$ let $Z^r_{dis}$ be the subset of $Z^r$ consisting of $r$-tuples of pairwise distinct points. For each such tuple $(p_1, \ldots, p_r)$ let $I(\{p_1, \ldots, p_r\})$ be the unique $B(Z)$ saturated ideal of $S[Z]$ defining the reduced subscheme $\{p_1, \ldots, p_r\}$ of $Z$ (see \cite[Cor.~3.8]{Cox95}). Let 
\[
Z^r_{gen} = \{(p_1, \ldots, p_r) \in Z^r_{dis}  \mid \frac{S[Z]}{I(\{p_1,\ldots, p_r\})} \text{ has Hilbert function }h_{r,Z}\}.
\]
The subset $Z^r_{gen}$ is open in $Z^r$ by \cite[Thm.~1.4]{BB21}. In particular it has a natural scheme structure. We use the following fact whose proof we present in Appendix~\ref{app:1}: there is a morphism $\psi_{r,Z}\colon Z^r_{gen} \to \Hilb_{S[Z]}^{h_{r,Z}}$ that on closed points maps $(p_1, \ldots, p_r)$ to $[I(\{p_1, \ldots, p_r\})]$ (Proposition~\ref{prop:existance_of_morphism}).

Consider the product morphism $f^r\colon X^r\to Y^r$. Since $X^r_{gen}$ is open in $X^r$ (see \cite[Thm.~1.4]{BB21}) its image is constructible  by Chevalley's theorem \cite[Thm.~10.20]{GW10}. Since $f^r$ is projective and surjective and $X^r_{gen}$ is a dense open subset of $X^r$ we conclude that $f^r(X^r_{gen})$ is a constructible dense subset of $Y^r$. Thus, there is an open dense subset $U\subseteq Y^r$ contained in $f^r(X^r_{gen})$ (see \cite[Ex.~II.3.18]{Har77}). Let $V=U\cap Y^r_{gen}$ and $W=(f^{r})^{-1}(V)\cap X^r_{gen}$. Consider the following diagram
\begin{center}
\begin{tikzcd}
\Hilb_{S[X]}^{h_{r,X}} \arrow[r, "{\pi}"] & \Hilb_{S[Y]}^{h_{r,Y}} \\
W \arrow[u, "\psi_{r,X}|_W"] \arrow[r, "f^r|_W"] & V \arrow[u, "\psi_{r,Y}|_V"].
\end{tikzcd}
\end{center}
We claim that this diagram is commutative. Let $(p_1, \ldots, p_r) \in W$. We have 
\[
\psi_{r,Y}\circ f^r(p_1, \ldots, p_r) = I(\{f(p_1), \ldots, f(p_r)\}) \text{ and }\pi\circ \psi_{r,X}(p_1, \ldots, p_r) = \varphi^{-1}(I(\{p_1,\ldots, p_r\})).
\]
Let $R \subset X$ be the reduced subscheme $\{p_1, \ldots, p_r\}$ and $i\colon R\to X$ be the closed immersion. Let $R'$ be the scheme-theoretic image of $R$. Since $R$ is reduced, $R'$ is the (reduced) subscheme $\{f(p_1), \ldots, f(p_r)\}$. The ideal sheaf of $R'$ is $\operatorname{ker} (\mathcal{O}_Y \to f_*\mathcal{O}_X \to f_*i_*\mathcal{O}_R)$. Since $\mathcal{O}_Y \to f_*\mathcal{O}_X$ is an isomorphism and $f_*$ is left-exact, the ideal sheaf of $R'$ is the pushforward of the ideal sheaf of $R$.
It follows from \cite[Thm.~3.5]{Man17}, Lemma~\ref{lem:unique_saturated_ideal} and Lemma~\ref{lem:connected_fibers_properties} that $\varphi^{-1}(I(\{p_1,\ldots, p_r\}))$ defines the subscheme $\{f(p_1), \ldots, f(p_r)\}$. Since this ideal is $B(Y)$-saturated  by Proposition~\ref{prop:main_observation}~\ref{it:prop_p2}, it coincides with $I(\{f(p_1), \ldots, f(p_r)\})$. This shows that the diagram indeed commutes.

By construction, we have $f^r(W) = V$ so it is dense in $Y^r_{gen}$. Since $\pi$ is projective, it follows that
\[
\Slip_{r,Y}  = \overline{\psi_{r,Y}\circ f^r(W)} = \overline{\pi\circ \psi_{r,X}(W)} = \pi(\overline{\psi_{r,X}(W)}) = \pi(\Slip_{r,X}).
\]
\end{proof}

To be able to use Theorem~\ref{thm:toric_fibration} it is necessary to construct a lift $\varphi\colon S[Y]\to S[X]$ of $f\colon X\to Y$. Assume that $f$ is a toric morphism. The pullback $\phi\colon \Pic(Y)\to \Pic(X)$ can be computed using \cite[Prop.~6.2.7]{CLS11}. Then for a graded homomorphism of graded rings $\varphi\colon S[Y]\to S[X]$ it can be checked with the help of Lemma~\ref{lem:characterization_of_property_2} if it satisfies property~\ref{it:deflift2} of the definition of a lift. Also Lemma~\ref{lem:char_of_property_3} can be used to verify if property~\ref{it:deflift3} is fulfilled.  
We illustrate this method in Example~\ref{exa:h1}. Similarly, one may verify that if $f\colon X\times Y \to X$ is the projection then the natural inclusion $\varphi\colon S[X]\to S[X\times Y]$ is its lift.
Therefore, we obtain the following consequence of Theorem~\ref{thm:toric_fibration}.

\begin{corollary}\label{cor:projection_criterion}
    Assume that $X$ and $Y$ are smooth projective complex toric varieties and let $f\colon X\times Y \to X$ be the projection. 
    Then $S[X]$ is a subring of $S[X\times Y]$ and for every positive integer $r$ the map $[I]\mapsto [I|_{S[X]}]$ defines a morphism of schemes $\Hilb_{S[X\times Y]}^{h_{r, X\times Y}} \to \Hilb_{S[X]}^{h_{r,X}}$ that maps $\Slip_{r,X\times Y}$ onto $\Slip_{r, X}$.  
\end{corollary}

\subsection{Projective embedding criterion}\label{sub:embedding}
\begin{theorem}\label{thm:embd_criterion}
Let $X$ be a smooth projective complex toric variety. Fix $[L]\in \Eff(X)$ and let $S^L\subseteq S[X]$ be the subring $\bigoplus_{d=0}^\infty S[X]_{[dL]}$. Let $r$ be a positive integer. Choose a basis $(g_0, \ldots, g_k)$ of $S^L_1$ and let $\varphi\colon \CC[t_0,\ldots, t_k] \to S[X]$ be given by $t_i\mapsto g_i$.
If the following conditions hold:
\begin{enumerate}[label=(\roman*)]
\item\label{it:thm_em_1} the ring $S^L = \bigoplus_{d=0}^\infty S[X]_{[dL]}$ is generated as a $\CC$-algebra by $S^L_1$;
\item\label{it:thm_em_3} if $h^0(X, \mathcal{O}_X(dL)) < r$, then $h^0(X, \mathcal{O}_X(dL)) = h^0(\PP^k, \mathcal{O}_{\PP^k}(d))$;
\item\label{it:thm_em_2} $B(X)\subseteq \sqrt{(S[X]_L)}$;
\end{enumerate}
then the map $[I]\mapsto [\varphi^{-1}(I)]$ defines a morphism $\pi\colon \Hilb_{S[X]}^{h_{r,X}} \to \Hilb_{S[\PP^k]}^{h_{r,\PP^k}}$ that maps $\Slip_{r,X}$ into $\Slip_{r,\PP^k}$.
\end{theorem}
\begin{proof}
Let $\phi\colon \mathbb{Z}\to \Pic(X)$ be the map given by $d\mapsto [dL]$. By assumption~\ref{it:thm_em_1} and Proposition~\ref{prop:main_observation}\ref{it:prop_p1} the morphism $\pi$ is well-defined if 
$h_{r,X}\circ \phi = h_{r, \PP^k}$. This equality of Hilbert functions follows from assumption~\ref{it:thm_em_3}.
The homomorphism $\varphi$ maps the irrelevant ideal of $\PP^k$ onto $S^L_{\geq 1}$. Therefore, by assumption~\ref{it:thm_em_2} and Proposition~\ref{prop:main_observation}\ref{it:prop_p2}, $\pi$ maps $\Slip_{r,X}$ into $\Slip_{r,\PP^k}$.
\end{proof}

\begin{example}\label{exa:p1_p1_is_reducible}
We show that $\Hilb_{S[\PP^1\times \PP^1]}^{h_{r, \PP^1\times \PP^1}}$ is reducible for every $r\geq 4.$ 

Let $S=\CC[\alpha_0, \alpha_1, \beta_0, \beta_1]$ be the Cox ring of $\PP^1\times \PP^1$. Here $\deg(\alpha_0) = \deg(\alpha_1) = (1,0)$ and $\deg(\beta_0) = \deg(\beta_1) = (0,1)$. Let $J = (\beta_0, \alpha_0^r)$.
We start with constructing an ideal $[I]\in \Hilb^{h_{r, \PP^1\times \PP^1}}_{S}$ with $I^{\sat} = J$.

Consider the lex monomial order with $\beta_0 > \beta_1 > \alpha_0 > \alpha_1$. For every $(a,b) \in \NN^2$ with $(a+1)(b+1) > r$ let $\mathcal{M}_{(a,b)}$ be the set of $(a+1)(b+1) - r$ smallest monomials in $J_{(a,b)}$.
We define
\[
I_{(a,b)} = \begin{cases}
\langle \mathcal{M}_{(a,b)} \rangle & \text{ if } (a+1)(b+1) > r\\
0  & \text{ otherwise}.
\end{cases}
\]
By construction $\dim_\CC (S/I)_{(a,b)} = h_{r, \PP^1\times \PP^1}(a,b)$, $I\subseteq J$  and  $I_{(a,b)} = J_{(a,b)}$ for every $a\geq r-1$. We claim that $I$ is an ideal. Take a monomial $M\in I_{(a,b)}$ and a monomial $N\in S_{(c,d)}$. Since $M\in I_{(a,b)}$ we have $M\in \mathcal{M}_{(a,b)}$ and as a consequence there are at least $r$ monomials in $J_{(a,b)}$ larger  than $M$. 
If we multiply any of them by $N$ we obtain a monomial in $J_{(a+c, b+d)}$ larger than $MN$. Thus $MN\in \mathcal{M}_{(a+c, b+d)} \subseteq I_{(a+c, b+d)}$. This finishes the proof that $[I]$ is a point of $\Hilb^{h_{r, \PP^1\times \PP^1}}_{S}$.

We verify that $[I]\notin \Slip_{r, \PP^1\times \PP^1}$ using Theorem~\ref{thm:embd_criterion}.
Choose $L=(1,r)$. Assumption~\ref{it:thm_em_1} of Theorem~\ref{thm:embd_criterion} is satisfied. Furthermore, $h^0(\PP^1\times \PP^1, \mathcal{O}(dL)) < r$ if and only if $d=0$, so the second assumption also holds. Finally, $(S[\PP^1\times \PP^1]_{L})$ contains $S[\PP^1\times \PP^1]_{(r,r)}$ so $B(\PP^1\times \PP^1) =  (S[\PP^1\times \PP^1]_{(1,1)}) \subseteq \sqrt{(S[\PP^1\times \PP^1]_{L})}$.

Let $(g_1, \ldots, g_{2r+2})$ be all monomials in $S[\PP^1\times \PP^1]_{L}$ written from the largest to the smallest in the same monomial order as before and consider the homomorphism $\CC[t_1, \ldots, t_{2r+2}] \to S[\PP^1\times \PP^1]$ given by $t_i\mapsto g_i$. We show that $[\varphi^{-1}(I)] \notin \Slip_{r, \PP^{2r+1}}$.

Let $K = \varphi^{-1}(I)$. We claim that $K^{\sat} = (t_1, \ldots, t_{2r}, t^r_{2r+1})$. Let $\mathfrak{m} = (t_1, \ldots, t_{2r+2})$.
For $i\in \{1,\ldots, 2r\}$ we have 
$
\varphi(t_i\cdot \mathfrak{m}^{r-1}) \subseteq (\beta_0)_{(r,r^2)} \subseteq (\beta_0, \alpha_0^r)_{(r,r^2)} = I_{(r,r^2)}.
$
It follows that $t_i\mathfrak{m}^{r-1}\subseteq K$. Furthermore, $\varphi(t^r_{2r+1}) \subseteq (\alpha_0^r)_{(r,r^2)} \subseteq I_{(r,r^2)}$ so  $(t_1, \ldots, t_{2r}, t^r_{2r+1})\subseteq K^{\sat}$. Both these ideals are saturated and have Hilbert polynomial $r$ so they coincide.

We have $t_1^{r-2} \notin K$ since $\varphi(t_1^{r-2}) = \beta_0^{r(r-2)}\alpha_0^{r-2}$ 
which by construction of $I$ is the unique monomial in $(\beta_0, \alpha_0^r)$ of degree $(r-2, r(r-2))$ that is not in $I$.
It follows from \cite[Thm.~2.7]{Man20} that $[K]\notin \Slip_{r, \PP^{2r+1}}$ and as a result $[I]\notin \Slip_{r, \PP^1\times \PP^1}$ by Theorem~\ref{thm:embd_criterion}.
\end{example}

\section{Projection from a product}\label{sec:projections}

For a smooth complex projective toric variety $Z$ with Cox ring $S[Z]$ and a divisor class $[D]\in \Eff(Z)$ we denote by $\mathcal{M}([D])$ be
 the set of all monomials in $S[Z]$ of degree $[D]$. For a positive integer $r$ we denote by $\mathcal{S}([D],r)$ the set of $r$ smallest
  monomials in $\mathcal{M}([D])$. We define $\mathcal{L}([D],r)$ to be the vector subspace of $S[Z]_{[D]}$ spanned by monomials from 
  $\mathcal{M}([D])\setminus \mathcal{S}([D],r)$. 

Assume that $X$ and $Y$ are complex smooth projective toric varieties with Cox rings $S[X]$ and $S[Y]$. Let $I_X\subseteq S[X]$ be a homogeneous ideal such that $S[X]/I_X$ has Hilbert function $h_{r, X}$. Fix any monomial orders on $S[X]$ and $S[Y]$. We consider $S[X\times Y]$ with the following product monomial order
\[
\alpha^{\mathbf{u}}\beta^{\mathbf{v}} > \alpha^{\mathbf{u}'}\beta^{\mathbf{v}'} \text{ if and only if } \beta^{\mathbf{v}} >  \beta^{\mathbf{v}'} \text{ or } \beta^{\mathbf{v}} = \beta^{\mathbf{v}'} \text{ and } \alpha^{\mathbf{u}} > \alpha^{\mathbf{u}'}
\]
where $\alpha^\mathbf{u}$ and $\alpha^{\mathbf{u}'}$ are monomials in $S[X]$ and $\beta^{\mathbf{v}}$ and $\beta^{\mathbf{v}'}$ are monomials in $S[Y]$.

We divide $\Eff(X\times Y)$ into subsets $\mathcal{A}$, $\mathcal{B}$ and $\mathcal{C}$ in the following way
\[
\mathcal{A} = \{([D],[E])\in \Eff(X\times Y) \mid  h_{r,X}([D]) = r\}
\]
\[
\mathcal{B} = \{([D],[E])\in \Eff(X\times Y) \mid  h_{r,X}([D]) < r \text{ and } h_{r,X\times Y}([D],[E]) =r \}
\]
and 
\[
\mathcal{C} = \{([D],[E])\in \Eff(X\times Y) \mid  h_{r,X\times Y}([D],[E])  < r \}.
\] 

We define the graded vector subspace $J = \bigoplus_{[D]\in \Eff(X), [E]\in \Eff(Y)} J_{([D],[E])}$ of $S[X\times Y]$ in the following way
\[
J_{([D],[E])} = \begin{cases}
\mathcal{L}([E],1)\cdot S[X]_{[D]} + S[Y]_{[E]}\cdot (I_X)_{[D]} & \text{ if } ([D], [E])\in \mathcal{A}\\
\mathcal{L}(([D],[E]),r) & \text{ if } ([D], [E])\in \mathcal{B}\\
0 & \text{ if } ([D], [E])\in \mathcal{C}.
\end{cases}
\]

We claim that 
\begin{enumerate}[label = (\alph*)]
\item $J$ is an ideal;
\item $J_{([D],\mathbf{0})} = (I_X)_{[D]}$ for every $[D]\in \Eff(X)$;
\item $H_{S[X\times Y]/J} = h_{r,X\times Y}$.
\end{enumerate}

\begin{lemma}\label{lem:is_and_ideal}
The graded vector space $J$ is an ideal of $S[X\times Y]$.
\end{lemma}
\begin{proof}

We show that $J_{([D_1], [E_1])}\cdot S[X\times Y]_{([D_2],[E_2])} \subseteq J_{([D_1+D_2], [E_1+E_2])}$ for every $([D_1], [E_1]), ([D_2], [E_2]) \in \Eff(X\times Y)$. We consider four cases:

\noindent\underline{Case 1: $([D_1], [E_1])\in \mathcal{A}$}

We have $([D_1+D_2], [E_1+E_2])\in \mathcal{A}$ and $J_{([D_1], [E_1])}\cdot S[X\times Y]_{([D_2], [E_2])} =$
\[
 = \mathcal{L}([E_1], 1)\cdot S[X]_{[D_1]}\cdot S[X\times Y]_{([D_2,] [E_2])} + S[Y]_{[E_1]}\cdot S[X\times Y]_{([D_2], [E_2])}\cdot (I_{X})_{[D_1]}
\]
\[
\subseteq \mathcal{L}([E_1],1)\cdot S[Y]_{[E_2]} \cdot S[X]_{[D_1+D_2]} + S[Y]_{[E_1]+[E_2]}\cdot S[X]_{[D_2]}\cdot (I_X)_{[D_1]}
\]
\[
\subseteq \mathcal{L}([E_1+E_2], 1)\cdot S[X]_{[D_1+D_2]} + S[Y]_{[E_1]+[E_2]}\cdot (I_{X})_{[D_1]+[D_2]} = J_{([D_1+D_2], [E_1+ E_2])}.
\]

\noindent \underline{Case 2:  $([D_1], [E_1])\in \mathcal{B}$ and  $([D_1+D_2], [E_1+E_2]) \in \mathcal{A}$}

Let $M$ be the smallest monomial in $S[Y]_{[E_1]}$. By the definition of $\mathcal{B}$ we have $h^0(X, \mathcal{O}_X(D_1)) <r $.  Hence, it follows from the definition of the monomial order in $S[X\times Y]$ that no monomial in $\mathcal{L}(([D_1], [E_1]), r)$ is divisible by $M$. Consequently, $\mathcal{L}(([D_1], [E_1]), r) \subseteq \mathcal{L}([E_1], 1)\cdot S[X]_{[D_1]}$.
Therefore
\[
J_{([D_1], [E_1])} \cdot S[X\times Y]_{([D_2], [E_2])} \subseteq \mathcal{L}([E_1], 1)\cdot S[X]_{[D_1]} \cdot S[X\times Y]_{([D_2], [E_2])}
\]
which is contained in $\mathcal{L}([E_1+E_2], 1)\cdot S[X]_{[D_1]+[D_2]} \subseteq J_{([D_1+D_2], [E_1+E_2])}.$

\noindent \underline{Case 3: $([D_1], [E_1])\in \mathcal{B}$ and  $([D_1+D_2], [E_1+E_2]) \in \mathcal{B}$}
\[
J_{([D_1], [E_1])}\cdot S[X\times Y]_{([D_2], [E_2])} = \mathcal{L}(([D_1], [E_1]),r)\cdot S[X\times Y]_{([D_2], [E_2])}. 
\]
Let $M\in \mathcal{L}(([D_1], [E_1]),r)$ and $M' \in \mathcal{M}(([D_2], [E_2]))$. For every $M''\in \mathcal{S}(([D_1], [E_1]),r)$ we have $M''M'  < MM'$. It follows that $MM' \in \mathcal{L}(([D_1+D_2], [E_1+E_2]),r)$ and therefore 
\[
J_{([D_1], [E_1])}\cdot S[X\times Y]_{([D_2], [E_2])} \subseteq  J_{([D_1+D_2], [E_1+E_2])}.
\]

\noindent\underline{Case 4: $([D_1], [E_1])\in \mathcal{C}$}
\[
J_{([D_1],[E_1])}\cdot S[X\times Y]_{([D_2], [E_2])} = 0\cdot S[X\times Y]_{([D_2], [E_2])} \subseteq J_{([D_1+D_2], [E_1+E_2])}. 
\]
\end{proof}

\begin{lemma}\label{lem:correct_HF_and_restricion}
The Hilbert function of $S[X\times Y]/J$ is $h_{r, X\times Y}$ and for every $[D]\in \Eff(X)$ we have $J_{([D],\mathbf{0})} = (I_X)_{[D]}$.
\end{lemma}
\begin{proof}
If $([D], [E])$ belongs to $\mathcal{B}$ or $\mathcal{C}$ then $J_{([D], [E])}$ has the right codimension in $S[X\times Y]_{([D], [E])}$. Assume that $([D], [E]) \in \mathcal{A}$.
The dimension of $J_{([D], [E])}$ is
\[
\dim_\CC \left(\mathcal{L}([E], 1)\cdot S[X]_{[D]}\right) + \dim_\CC \left( S[Y]_{[E]}\cdot (I_X)_{[D]} \right)  - \dim_\CC \left(\mathcal{L}([E], 1)\cdot (I_X)_{[D]}\right)
\]
\[
= (h^0(\mathcal{O}_Y(E))-1)h^0(\mathcal{O}_X(D)) + h^0(\mathcal{O}_Y(E))(h^0(\mathcal{O}_X(D))-r) - (h^0(\mathcal{O}_Y(E))-1)(h^0(\mathcal{O}_X(D))-r)
\]
which gives $\dim_\CC (S[X\times Y]/J)_{([D],[E])} = r$.

If  $([D],\mathbf{0})\in \mathcal{C}$ then $J_{([D], \mathbf{0})} = 0 = (I_X)_{[D]}$. If $([D], \mathbf{0}) \in \mathcal{A}$ then $J_{([D],\mathbf{0})} = (I_X)_{[D]}$ since $\mathcal{L}(\mathbf{0}, 1) = 0$ and $S[Y]_\mathbf{0} = \CC\cdot 1$.
\end{proof}

\begin{proposition}\label{prop:projections_from_product_are_surjective}
Let $X$ and $Y$ be smooth projective complex toric varieties and let $r$ be a positive integer. The natural map
$\Hilb^{h_{r,X\times Y}}_{S[X\times Y]} \to \Hilb^{h_{r,X}}_{S[X]}$ given by $[K]\mapsto [K|_{S[X]}]$ is surjective.
\end{proposition}
\begin{proof}
Let $[I]\in \Hilb^{h_{r,X}}_{S[X]}$ and let $J \subseteq S[X\times Y]$ be the graded vector space constructed as in the beginning of this section.
It follows from Lemmas~\ref{lem:is_and_ideal}~and~\ref{lem:correct_HF_and_restricion} that $[J] \in \Hilb^{h_{r, X\times Y}}_{S[X\times Y]}$ and $J|_{S[X]} = I$.
\end{proof}

\section{Surjectivity of the natural map to the Hilbert scheme}\label{sec:surjectivity_in_general}
Let $S=\CC[\alpha_0,\ldots, \alpha_{n}]$ be the standard $\ZZ$-graded polynomial ring. Consider its graded dual ring $S^* = \CC_{dp}[x_0, \ldots, x_n]$. The subscript $dp$ refers to the divided power structure.
See \cite{IK06}[App.~A] for basic properties of $S^*$.
We fix a monomial order $<$ on $S$ and we consider the analogous order of monomials in $S^*$. Recall that we have the apolarity action of $S$ on $S^*$ given on monomials by $\alpha_i \lrcorner \mathbf{x}^{\mathbf{u}} = \begin{cases} 0 & \text{ if } u_i = 0 \\ \mathbf{x}^{\mathbf{u}-\mathbf{e}_i} & \text{ otherwise}\end{cases}$.

\begin{lemma}\label{lem:largest_monomial_lemma}
Let $k$ be a positive integer. Assume that $V_k\subseteq S^*_{k}$ is a proper vector subspace. If $\mathbf{x}^\mathbf{u}\in S^*_{k}$ is the largest monomial in $S^*_{k}\setminus V_{k}$, then
\begin{enumerate}[label=(\alph*)]
\item\label{it:lemma_pt1} there exists at most one $i\in \{0,1, \ldots,n\}$ with $\alpha_i \lrcorner \mathbf{x}^\mathbf{u} \notin S_1\lrcorner V_{k}$;
\item if $\alpha_i \lrcorner \mathbf{x}^\mathbf{u} \notin S_1\lrcorner V_k$, then $\alpha_i \lrcorner \mathbf{x}^\mathbf{u}$ is the largest monomial in $S^*_{k-1}\setminus (S_1\lrcorner V_{k})$.
\end{enumerate}
\end{lemma}
\begin{proof}
\hfill
Assume that there exists $\alpha_j \notin \Ann(\mathbf{x}^\mathbf{u})$ and let $\alpha_i$ be the largest among them.
Assume that $\alpha_j$ is another form that does not annihilate $\mathbf{x}^\mathbf{u}$. We have $x_j< x_i$ so the monomial $\frac{x_i}{x_j}\cdot \mathbf{x}^{\mathbf{u}}$ is greater than $\mathbf{x}^{\mathbf{u}}$ and is therefore in $V_k$. It follows that
$\alpha_j \lrcorner \mathbf{x}^\mathbf{u} =  \alpha_i \lrcorner (\frac{x_i}{x_j}\cdot \mathbf{x}^{\mathbf{u}}) \in S_1\lrcorner V_k.$ This establishes part~\ref{it:lemma_pt1}.

With the above notation, let $\mathbf{x}^{\mathbf{w}}$ be a monomial in $S^*_{k-1}$ that is greater than $\alpha_i \lrcorner \mathbf{x}^{\mathbf{u}}$.
Then $\mathbf{x}^{\mathbf{u}} = (\alpha_i \lrcorner \mathbf{x}^{\mathbf{u}})\cdot x_i < \mathbf{x}^{\mathbf{w}}\cdot x_i$, so the latter monomial is in $V_k$. As a result, $\mathbf{x}^\mathbf{w} = \alpha_i \lrcorner ( \mathbf{x}^{\mathbf{w}}\cdot x_i) \in S_1 \lrcorner V_k$.
\end{proof}

\begin{corollary}\label{cor:adding_greatest_monomials_gives_an_ideal}
Let $k$ be a positive integer. Assume that $V_k \subseteq S^*_k$ and $V_{k-1}\subseteq S^*_{k-1}$ are vector subspaces.
For $\bullet\in \{k-1, k\}$ 
let $W_\bullet = V_\bullet$ if $V_{\bullet} = S^*_{\bullet}$ and otherwise let $W_{\bullet} = V_{\bullet} + \langle F_{\bullet}\rangle$ where $F_{\bullet}$ is the largest monomial in $S^*_{\bullet}\setminus V_{\bullet}.$
If $S_{1}\lrcorner V_{k}\subseteq V_{k-1}$, then $S_{1}\lrcorner W_{k}\subseteq W_{k-1}$.
\end{corollary}
\begin{proof}
If $V_k = W_k$ or $V_{k-1} = W_{k-1}$ the claim is trivially true.
Assume that both $V_{k}$ and $V_{k-1}$ are proper subspaces of $S^*_k$ and $S^*_{k-1}$, respectively.

It is enough to show that $S_{1}\lrcorner F_k \subseteq W_{k-1}$. There are two possibilities. Either $S_{1}\lrcorner F_k \subseteq (S_1\lrcorner V_k)$ or not. In the first case $S_{1}\lrcorner F_k$ is contained in $V_{k-1} \subseteq W_{k-1}$. In the latter case, by Lemma~\ref{lem:largest_monomial_lemma}  there exists a unique $\alpha_i \in S_{1}$ for which $\alpha_i \lrcorner F_{k} \notin (S_1\lrcorner V_{k})$.
Furthermore, by the same lemma $\alpha_i \lrcorner F_{k} = F_{k-1}$. It follows that $S_{1}\lrcorner F_{k} \subseteq W_{k-1}$.
\end{proof}

\begin{proposition}\label{prop:ideals_lift}
Assume that $J \subseteq S$ is a homogeneous ideal defining a zero-dimensional length $r$ subscheme of $\PP^n$ such that $H_{S/J}(k) \leq H_{S/J}(k+1)$ for every $k$. Let $a$ be the smallest integer for which $h_{r,\PP^n}(a) = r$ and let $b$ be the maximal integer for which $H_{S/J}(b) \neq r$.

For every $a\leq k\leq b$ let $n_k$ be the smallest integer such that if $\mathcal{M}_{k}= \{F_{k,1}, F_{k,2}, \ldots, F_{k,n_k}\}$ is the set of $n_k$ largest monomials in $S^*_k$, then $W_k= J^\perp_k + \langle \mathcal{M}_k \rangle$ is an $r$-dimensional linear space. 
The following hold
\begin{enumerate}[label=(\alph*)]
\item $I=\left(\bigoplus_{k=a}^b W_k^\perp\right) + J_{\geq b+1}$ is an ideal;
\item $H_{S/I} = h_{r,\PP^n}$;
\item $I_{\geq b+1} = J_{\geq b+1}$.
\end{enumerate}
\end{proposition}
\begin{proof}
We introduce some more notation. Let $h = h_{r,\PP^n} - H_{S/J}$. For every $a\leq k \leq b$ and every $0< s \leq h(k)$ let $1\leq n_{k,s} \leq n_k$ be the smallest integer such that 
\[
W_{k, s} = \langle F_{k,1}, \ldots, F_{k, n_{k,s}} \rangle + J_k^\perp
\]
has dimension $H_{S/J}(k) + s$.

In order to show that $I$ is an ideal it suffices to show that $S_1 \lrcorner W_k \subseteq W_{k-1}$ for every $a+1\leq k \leq b$.
We show by induction on $l$ that $S_1$ maps $W_{k, l}$ into $W_{k-1, l}$. 
Note that we use here the assumption that $H_{S/J}(k-1) \leq H_{S/J}(k)$ to guarantee that $W_{k-1, l}$ is defined for all $0<l\leq h(k)$.

The case $l=1$ follows from Corollary~\ref{cor:adding_greatest_monomials_gives_an_ideal}. Let $2\leq l \leq h(k)$ and assume that $S_1\lrcorner W_{k, l-1} \subseteq W_{k-1, l-1}$.
Application of Corollary~\ref{cor:adding_greatest_monomials_gives_an_ideal} with 
$(W_{k, l-1}, W_{k,l})$ and
 $(W_{k-1, l-1}, W_{k-1,l})$ playing the roles of $(V_{k}, W_{k})$ and $(V_{k-1}, W_{k-1})$, respectively, shows that $S_1 \lrcorner W_{k, l} \subseteq W_{k-1, l}$. 
We have proved that 
\[
S_{1}\lrcorner W_{k} = S_{1}\lrcorner W_{k, h(k)} \subseteq W_{k-1,h(k)} \subseteq W_{k-1, h({k-1})} = V_{k-1}.
\]
This finishes the proof that $I$ is an ideal.

By construction we have $I_k \subseteq J_k$ for every $k\leq b$. Therefore, $I_{\geq b+1}  = J_{\geq b+1}$.
In particular, $H_{S/I}(k) = h_{r, \PP^n}(k)$ for every $k\geq b+1$. By the definition of $I$ we have $H_{S/I}(k) = h_{r, \PP^n}(k)$ for every $k\leq b$.
\end{proof}

\begin{corollary}\label{cor:standard_map_is_surjective}
The natural map $\Hilb^{h_{r,\PP^n}}_{S} \to \mathcal{H}ilb_r(\PP^n)$ given by $[I]\mapsto [\Proj(S/I)]$ is (set-theoretically) surjective.
\end{corollary}
\begin{proof}
Let $[I] \in \Hilb^{h_{r,\PP^n}}_S$ and let $f_{r,\PP^n}$ be the Hilbert function of $S/I_{\geq r}$.
Using \cite[Lem.~4.1]{HS04}, it is enough to show that the map
$\Hilb^{h_{r,\PP^n}}_S \to \Hilb^{f_{r,\PP^n}}_S$ given by $[I]\mapsto [I_{\geq r}]$ is surjective.
Let $[K]\in \Hilb^{f_{r,\PP^n}}_S$ and let $J=K^{\sat}$. Let $b$ be the maximal integer $k$ for which $H_{S/J}(k) \neq r$.
By Proposition~\ref{prop:ideals_lift} there exists
$[I]\in \Hilb^{h_{r,\PP^n}}_S$ with $I_{\geq b+1} = J_{\geq b+1}$. We have $b+1 < r$ and $J_{\geq r} = K_{\geq r}$ so
$I_{\geq r} = K$.
\end{proof}

The following proposition describes all the cases when the multigraded Hilbert scheme parametrizing ideals of $r$ points in general position in $\PP^n$ is irreducible. The main observations were already done previously (see \cite{Man20} and \cite{JM22}).

\begin{proposition}\label{prop:classification_of_reducible_MGHS_pn}
The scheme $\Hilb^{h_{r,\PP^n}}_S$ is reducible if and only if $n\geq 2$ and $r\geq 4$.
\end{proposition}
\begin{proof}
If $n=1$, then $\Hilb^{h_{r, \PP^n}}_S$ is isomorphic to $\mathcal{H}ilb_r(\PP^1)$ by \cite[Lem.~4.1]{HS04}. Hence it is irreducible.

Assume that $r\leq 3$. By \cite[Prop.~3.1]{CEVV09} in order to show that $\Hilb_S^{h_{r, \PP^n}}$ is irreducible we may and do assume that $n\leq 2$. Let $[I]\in \Hilb^{h_{r, \PP^n}}_S$. 
Consider $[\Proj (S/I)]\in \mathcal{H}ilb_r(\PP^n)$. The natural map  from Corollary~\ref{cor:standard_map_is_surjective} induces a map $\Slip_{r, \PP^n}\to \mathcal{H}ilb^{sm}_r(\PP^n)$---the smoothable component of $\mathcal{H}ilb_r(\PP^n)$. This map is dominant and projective so is surjective. Since $r\leq 3$ and $n\leq 2$ we have $\mathcal{H}ilb^{sm}_r(\PP^n) = \mathcal{H}ilb_r(\PP^n)$. We conclude, that there is a point $[I']\in \Slip_{r, \PP^n}$ with $\Proj(S/I') = \Proj(S/I)$. Since $r\leq 3$, $I'=I$, so $[I]\in \Slip_{r, \PP^n}$.

Assume that $n\geq 2$ and $r\geq 4$. Let $J=(\alpha_0, \ldots, \alpha_{n-2}, \alpha_{n-1}^r)$. We claim that there exists a point $[I]$ in 
$\Hilb^{h_{r, \PP^n}}_S\setminus \Slip_{r, \PP^n}$ with $[\Proj(S/I)] = [\Proj(S/J)]$.
Let $I$ be the ideal constructed in the proof of Proposition~\ref{prop:ideals_lift}. By construction $\alpha_0^{r-2} \notin I$. Therefore, by \cite[Thm.~2.7]{Man20}, $[I]\notin \Slip_{r, \PP^n}$.
\end{proof}

\section{Tangent space criterion}
Let $X$ be a smooth projective complex toric variety and $S$ be its Cox ring. Fix a positive integer $r$. We denote by $\mathcal{C}(r,X)$ the set of all $[D]\in \Eff(X)$ with $\dim_\CC S_{[D]} \geq r$.
For any graded vector subspace $V\subseteq S$ and any subset $\mathcal{D} \subseteq \Pic(X)$ we denote by $V_{\mathcal{D}}$ the graded vector subspace $V_{\mathcal{D}} = \bigoplus_{[D]\in \mathcal{D}}V_{[D]}$ of $S$.

\begin{definition}\label{def:ssi}
A subset $\mathcal{E}\subseteq \Eff(X)$ is called $(r,X)$-sufficient if for every $B(X)$-saturated ideal $[I]\in \Hilb_{S}^{h_{r,X}}$ we have $I = ((I_{\mathcal{E}})\colon B(X)^\infty)$.
\end{definition}

If $\mathcal{E} \subseteq \Eff(X)$ is as in Definition~\ref{def:ssi} then every $B(X)$-saturated ideal $I$ with Hilbert function of the quotient algebra equal to $h_{r,X}$ can be reconstructed from $I_{\mathcal{E}}$. This means, that we can alter all the other graded parts of $I$ and we do not lose the information that we started from $I$. This observation is exploited in Proposition~\ref{prop:cg_subsets_give_necessary_conditions}.

\begin{proposition}\label{prop:cg_subsets_give_necessary_conditions}
Let $\mathcal{A} \subseteq \mathcal{B} \subseteq \Eff(X)$ be subsets such that 
\begin{enumerate}[label=(\roman*)]
    \item\label{con1:prop_ts} $\mathcal{B} + \Eff(X) \subseteq \mathcal{B}$
    \item\label{con2:prop_ts} $\mathcal{A} + \Eff(X) \subseteq \mathcal{A}$
    \item\label{con3:prop_ts} $\mathcal{B}\setminus \mathcal{A}$ is $(r,X)$-sufficient.
\end{enumerate}
If $[I]\in \Hilb_{S}^{h_{r,X}}$ and $J = I_\mathcal{B} + S_{\mathcal{A}}$, then
\begin{enumerate}[label=(\alph*)]
    \item\label{part1:prop_ts} $J$ is a homogeneous ideal
    \item\label{part2:prop_ts} if $g$ is the Hilbert function of $S/J$ then there is a morphism $\pi\colon \Hilb^{h_{r,X}}_{S} \to \Hilb^g_{S}$ given on closed points by $[I']\mapsto [I'_\mathcal{B} + S_{\mathcal{A}}]$;
    \item\label{part3:prop_ts} $\pi$ is injective on the set of $B(X)$-saturated ideals;
    \item\label{part4:prop_ts} if $[I]$ does not lie on any irreducible component of $\Satbar^{h_{r,X}}$ whose dimension is at most equal to  $\dim_\CC \Hom_S(J,S/J)_\mathbf{0}$, then $[I]$ is not in $\Satbar^{h_{r,X}}$;
    \item\label{part5:prop_ts} if $\dim_\CC \Hom_S(J,S/J)_\mathbf{0} < r\cdot \dim X$, then $[I]$ is not in $\Slip_{r,X}$.
\end{enumerate}
\end{proposition}
\begin{proof}
Parts~\ref{part1:prop_ts}~and~\ref{part2:prop_ts} follow from conditions~\ref{con1:prop_ts}~and~\ref{con2:prop_ts}.
Part~\ref{part3:prop_ts} is a consequence of \ref{con3:prop_ts}. Finally, the last two parts follow from part~\ref{part3:prop_ts} since 
$\Hom_S(J,S/J)_\mathbf{0}$ is identified with the tangent space to $\Hilb^g_S$ at $[J]$ (see \cite[Prop.~1.6]{HS04}) and a general point of every irreducible component of $\Satbar^{h_{r,X}}$ corresponds to a $B(X)$-saturated ideal.
\end{proof}

To make Proposition~\ref{prop:cg_subsets_give_necessary_conditions} more useful we identify some conditions on a subset $\mathcal{E} \subseteq \Eff(X)$ that guarantee that $\mathcal{E}$ is $(r,X)$-sufficient.

\begin{lemma}\label{lem:ifthensufficient}
If $\mathcal{E}\subseteq \Eff(X)$ is such that for every $[D]\in \mathcal{C}(r,X)$ there exist $[E]\in \mathcal{E}\cap \mathcal{C}(r,X)$, $k\in \NN$, $[F],[G]\in \Nef(X)$ such that the following hold
\begin{enumerate}[label=(\roman*)]
    \item\label{ass1:lemma} $[E+F] \in \mathcal{E}\cap \mathcal{C}(r,X)$;
    \item\label{ass2:lemma} $[E+kF] = [D+G]$;
    \item\label{ass3:lemma} the multiplication map $S_{[F]}\otimes S_{[E+lF]} \to S_{[E+(l+1)F]}$ is surjective for every non-negative integer $l$;
\end{enumerate}
then $\mathcal{E}$ is $(r,X)$-sufficient.
\end{lemma}
\begin{proof}
Let $[I] \in \Hilb_S^{h_{r,X}}$ be a $B(X)$-saturated ideal and $J = ((I_{\mathcal{E}})\colon B(X)^\infty)$.
We have $J \subseteq I$ and for every $[D]\in \mathcal{E}$ we have $J_{[D]} = I_{[D]}$.

Let $[D]\in \mathcal{C}(r,X)$ and let $k, [E],[F]$ and $[G]$ satisfy \ref{ass1:lemma}--\ref{ass3:lemma}. By the above, we have $H_{S/J}([E]) = H_{S/J}([E+F]) = r$. There is a nonzerodivisor $\ell$ on $S/J$ of degree $[F]$ (see \cite[Prop.~3.1]{MacSmi04}) so multiplication by $\ell$ defines a surjection $(S/J)_{[E]} \to (S/J)_{[E+F]}$. It follows by induction using assumption \ref{ass3:lemma} that $H_{S/J}([E+lF]) = r$ for every $l\in \mathbb{Z}_{\geq 0}$. In particular, $r=H_{S/J}([E+kF]) = H_{S/J}([D+G])$. Since there 
is a nonzerodivisor on $S/J$ of degree $[G]$ we conclude that $H_{S/J}([D])\leq r$.
We showed that $H_{S/J}([D]) \leq r$ for every $[D]\in\mathcal{C}(r,X)$. However, $J\subseteq I$ and $S/I$ has Hilbert function $h_{r,X}$ so we conclude that $J = I$.
\end{proof}

The approach to showing that $\mathcal{E}$ is $(r,X)$-sufficient presented in Lemma~\ref{lem:ifthensufficient} is based on the existence of nonzerodivisors on the quotient algebra of a $B(X)$-saturated ideal. Another possibility of finding $(r,X)$-sufficient subsets comes from the notion of the (multigraded) Castelnuovo-Mumford regularity.

\begin{remark}
    Suppose that $\mathcal{E} \subseteq \Eff(X)$ is a subset such that for every $B(X)$-saturated ideal $[I]\in \Hilb_{S}^{h_{r,X}}$ and any minimal generator $f$ of $I$ there exists a non-negative integer $k$ such that for every minimal generator $g$ of $B(X)^k$ we have $\deg(fg)\in \mathcal{E}$. Then $\mathcal{E}$ is $(r,X)$-sufficient. 
    
    Let $e = \min \{a\mid \dim_\CC S[\PP^n]_a \geq r\}$.  It follows from \cite[Thm.~4.2]{Eis05} that if $\mathcal{E}$ contains any degree greater than $e$, then $\mathcal{E}$ is $(r, \PP^n)$-sufficient.
    For more general toric varieties one may consider the multigraded Castelnuovo-Mumford regularity \cite{MacSmi04}\cite{MacSmi03},  to try to obtain analogous bounds on the degrees of minimal generators of saturated ideals.
\end{remark}

\begin{example}\label{exa:tsex11}
    Let $S = \CC[\alpha_0, \alpha_1, \alpha_2]$ and $B = (S_1)$ be the irrelevant ideal. Pick an ideal $[I]\in \Hilb_{S}^{h_{6,\PP^2}}$. We consider the Hilbert functions of three more related ideals.
\[
\begin{array}{|c|c|c|c|c|}
\hline
     \text{Ideal} & \text{Hilbert function} & \mathcal{A} & \mathcal{B}   & \mathcal{B}\setminus \mathcal{A} \\  \hline
            I\cap B^4 & (1,3,6,10,6,6,\ldots )& \emptyset  & \{4,5, \ldots\} & \{4,5,\ldots\}  \\
            I+B^5 & (1,3,6,6,6,0,0,0, \ldots)& \{5,6,\ldots\}  & \NN   & \{0,1,2,3,4\} \\
            I\cap B^4+B^6 & (1,3,6,10,6,6,0,0,\ldots) & \{6,7,\ldots \}  & \{4,5,\ldots\}  & \{4,5\} \\
            \hline
\end{array}
\]
We have $\mathcal{C}(6, \PP^2) = \{2,3,\ldots\}$. Using Lemma~\ref{lem:ifthensufficient} it can be shown that $\{3,4\}$ and $\{4,5\}$ are $(6, \PP^2)$-sufficient and thus so are $\mathcal{B}\setminus \mathcal{A}$ from all the rows.
If $[I]\in \Slip_{6, \PP^2}$ then using Proposition~\ref{prop:cg_subsets_give_necessary_conditions} with $\mathcal{B}$ and  $\mathcal{A}$ as in the table we conclude that $\dim_\CC \Hom_S(K,S/K)_0 \geq 12$ for $K \in \{I\cap B^4, I+B^5, I\cap B^4+B^6\}$.
Consider the ideal
\[
I = (\alpha_0^3, \alpha_0\alpha_1^2, \alpha_0^2\alpha_2, \alpha_0\alpha_1\alpha_2, \alpha_0\alpha_2^4, \alpha_1^6).
\]
We have
\[
\dim_\CC \Hom_S(I+B^5, S/(I+B^5))_{0} = 8 < 12 = \dim \Slip_{6, \PP^2}.
\]
Therefore, $[I]\notin \Slip_{6, \PP^2}$.
\end{example}

\begin{example}
    Let $S=\CC[\alpha_0, \alpha_1, \beta_0,\beta_1]$ be the Cox ring of $\PP^1\times \PP^1$. We define the ideals $\mathfrak{a} = (\alpha_0,\alpha_1)$, $\mathfrak{b} = (\beta_0,\beta_1)$ and $B=\mathfrak{a}\mathfrak{b}$. Pick an ideal $[I]\in \Hilb_{S}^{h_{2,\PP^1\times \PP^1}}$. We consider the Hilbert functions of four more related ideals. The corresponding subsets of $\Nef(\PP^1\times \PP^1) = \NN^2$ are as follows.
\begin{figure}[H]
 \begin{minipage}{.24\textwidth}
 \centering
  \begin{tikzpicture}[scale=0.75]
    \coordinate (Origin)   at (0,0);
    \coordinate (XAxisMin) at (0,0);
    \coordinate (XAxisMax) at (4.2,0);
    \coordinate (YAxisMin) at (0,0);
    \coordinate (YAxisMax) at (0,4.2);
    \draw [thin, black,-latex] (XAxisMin) -- (XAxisMax);
    \draw [thin, black,-latex] (YAxisMin) -- (YAxisMax);

    \clip (-0.3,-0.3) rectangle (4.3cm,4.3cm); 
    \coordinate (Bone) at (0,1);
    \coordinate (Btwo) at (1,0);
\foreach \x in{1,...,4} 
\draw[xshift=\x cm] (0,2pt) -- (0,-2pt) node[below,fill=white]{};
\foreach \y in{1,...,4} 
\draw[yshift=\y cm] (-2pt,0) -- (2pt,0) node[left]{};
    \foreach \x in {2,3,4}{
     \foreach \y in {0,1,2,3,4}{
         \node[draw,circle,inner sep=2pt, fill] at (\x,\y) {};
   }
   }
    \foreach \x in {0,1}{
     \foreach \y in {0,1,2,3,4}{
         \node[draw,circle,inner sep=2pt, fill, color=gray] at (\x,\y) {};
   }
   }
	 \end{tikzpicture}\caption{$I+\mathfrak{a}^2$}
\end{minipage}
 \begin{minipage}{.24\textwidth}
 \centering
  \begin{tikzpicture}[scale=0.75]
    \coordinate (Origin)   at (0,0);
    \coordinate (XAxisMin) at (0,0);
    \coordinate (XAxisMax) at (4.2,0);
    \coordinate (YAxisMin) at (0,0);
    \coordinate (YAxisMax) at (0,4.2);
    \draw [thin, black,-latex] (XAxisMin) -- (XAxisMax);
    \draw [thin, black,-latex] (YAxisMin) -- (YAxisMax);

    \clip (-0.3,-0.3) rectangle (4.3cm,4.3cm); 
    \coordinate (Bone) at (0,1);
    \coordinate (Btwo) at (1,0);

\foreach \x in{1,...,4} 
\draw[xshift=\x cm] (0,2pt) -- (0,-2pt) node[below,fill=white]{};
\foreach \y in{1,...,4} 
\draw[yshift=\y cm] (-2pt,0) -- (2pt,0) node[left]{};
   \foreach \x in {2,3,4}{
      \foreach \y in {1,2,3,4}{
        \node[draw,circle,inner sep=2pt, fill, color=gray] at (\x,\y) {};
    }
    }
	\end{tikzpicture}\caption{$I\cap \mathfrak{a}^2\mathfrak{b}$}
\end{minipage}
 \begin{minipage}{.24\textwidth}
 \centering
  \begin{tikzpicture}[scale=0.75]
    \coordinate (Origin)   at (0,0);
    \coordinate (XAxisMin) at (0,0);
    \coordinate (XAxisMax) at (4.2,0);
    \coordinate (YAxisMin) at (0,0);
    \coordinate (YAxisMax) at (0,4.2);
    \draw [thin, black,-latex] (XAxisMin) -- (XAxisMax);
    \draw [thin, black,-latex] (YAxisMin) -- (YAxisMax);

    \clip (-0.3,-0.3) rectangle (4.3cm,4.3cm); 
    \coordinate (Bone) at (0,1);
    \coordinate (Btwo) at (1,0);

\foreach \x in{1,...,4} 
\draw[xshift=\x cm] (0,2pt) -- (0,-2pt) node[below,fill=white]{};
\foreach \y in{1,...,4} 
\draw[yshift=\y cm] (-2pt,0) -- (2pt,0) node[left]{};
     \foreach \x in {0,1,2}{
       \foreach \y in {0,1,2}{
         \node[draw,circle,inner sep=2pt, fill, color=gray] at (\x,\y) {};
     }
     }
    \foreach \x in {3,4,5}{
       \foreach \y in {0,1,2,3,4,5}{
         \node[draw,circle,inner sep=2pt, fill, color=black] at (\x,\y) {};
     }
     }
    \foreach \x in {0,1,2}{
       \foreach \y in {3,4,5}{
         \node[draw,circle,inner sep=2pt, fill, color=black] at (\x,\y) {};
     }
     }
	\end{tikzpicture}\caption{$I + \mathfrak{a}^3+\mathfrak{b}^3$}
\end{minipage}
 \begin{minipage}{.24\textwidth}
 \centering
  \begin{tikzpicture}[scale=0.75]
    \coordinate (Origin)   at (0,0);
    \coordinate (XAxisMin) at (0,0);
    \coordinate (XAxisMax) at (4.2,0);
    \coordinate (YAxisMin) at (0,0);
    \coordinate (YAxisMax) at (0,4.2);
    \draw [thin, black,-latex] (XAxisMin) -- (XAxisMax);
    \draw [thin, black,-latex] (YAxisMin) -- (YAxisMax);

    \clip (-0.3,-0.3) rectangle (4.3cm,4.3cm); 
    \coordinate (Bone) at (0,1);
    \coordinate (Btwo) at (1,0);

\foreach \x in{1,...,4} 
\draw[xshift=\x cm] (0,2pt) -- (0,-2pt) node[below,fill=white]{};
\foreach \y in{1,...,4} 
\draw[yshift=\y cm] (-2pt,0) -- (2pt,0) node[left]{};
     \foreach \x in {1,2}{
      \foreach \y in {1,2}{
         \node[draw,circle,inner sep=2pt, fill, color=gray] at (\x,\y) {};
     }
     }
     \foreach \x in {3,4}{
      \foreach \y in {1,2,3,4}{
         \node[draw,circle,inner sep=2pt, fill, color=black] at (\x,\y) {};
     }
     }     
    \foreach \x in {1,2}{
      \foreach \y in {3,4}{
         \node[draw,circle,inner sep=2pt, fill, color=black] at (\x,\y) {};
     }
     }
	\end{tikzpicture}\caption{$(I+\mathfrak{a}^3+\mathfrak{b}^3)\cap B$}
\end{minipage}
\end{figure}
Dots correspond to $\mathcal{B}$. Black ones correspond to $\mathcal{A}$ and gray ones correspond to $\mathcal{B}\setminus \mathcal{A}$. Using Lemma~\ref{lem:ifthensufficient} it can be shown that $\mathcal{B}\setminus \mathcal{A}$ is $(2, \PP^1\times \PP^1)$-sufficient in each of the four cases. Therefore, it follows from Proposition~\ref{prop:cg_subsets_give_necessary_conditions} that if $[I]\in \Slip_{2, \PP^1\times \PP^1}$ then for every $K\in \{I+\mathfrak{a}^2, I\cap \mathfrak{a}^2\mathfrak{b}, I+\mathfrak{a}^3+\mathfrak{b}^3, (I+\mathfrak{a}^3+\mathfrak{b}^3)\cap B\}$ we have $\dim_\CC \Hom_S(K, S/K)_{(0,0)} \geq 4$.
\end{example}

\begin{example}\label{exa:hr}
Let $a\geq 1$ be an integer and $X=\mathcal{H}_a = \PP(\mathcal{O}_{\PP^1}\oplus \mathcal{O}_{\PP^1}(a))$ be the Hirzebruch surface. Its fan is the complete fan in $\mathbb{R}^2$ whose rays are spanned by $u_1=(1,0), u_2=(0,-1), u_3=(-1,a)$ and $u_4=(0,1)$.
If we choose the torus invariant divisors of $X$ corresponding to the rays $\operatorname{span}(u_3)$ and $\operatorname{span}(u_4)$ as a basis of $\Pic(X)\cong \ZZ^2$, then the Cox ring is $S=\CC[\alpha_1, \alpha_2, \alpha_3, \alpha_4]$ with $\deg(\alpha_1) = (1,0)$, $\deg(\alpha_2)  = (a,1)$, $\deg(\alpha_3)= (1,0)$ and $\deg(\alpha_4)=(0,1)$ (see \cite[Thm.~4.2.1]{CLS11}). 
Furthermore, $\Nef(X) = \NN (1,0) + \NN (a,1)$.

Let $I= (\alpha_1\alpha_3,\alpha_1\alpha_2, \alpha_1^a\alpha_4, \alpha_2^2)$, $\mathcal{A} = \{(u_1,u_2)\in \NN^2 \mid u_2\geq 2\}$ and $\mathcal{B}=\NN^2$.
We verify that $\mathcal{B}\setminus \mathcal{A}$ is $(2, \mathcal{H}_a)$-sufficient using Lemma~\ref{lem:ifthensufficient}. 
Let $D=(u_1,u_2)\in \mathbb{N}^2$. Choose $G=(d,0)\in \NN^2$ such that $d+u_1-au_2\geq 1$. Let $E=(d+u_1-au_2, 0)$ and $F=(a,1)$. We have $E, E+F\in \mathcal{C}(2,X)\cap \mathcal{B} \setminus \mathcal{A}$, $D+G = E+u_2F$ and $F,G \in \Nef(X)$. We are left with verifying that $S_F\otimes S_{E+lF} \to S_{E+(l+1)F}$ is surjective for every non-negative integer $l$. This follows from the following equality that holds for any $k\geq 0$ 
\[
S_{E+kF} = \langle \alpha_2^i\alpha_4^{k-i} \alpha_1^j\alpha_3^{d+u_1+(k-u_2-i)a - j}\rangle_{0\leq i \leq k, 0\leq j \leq d+u_1+(k-u_2-i)a}.
\]

Conditions~\ref{con1:prop_ts}~and~\ref{con2:prop_ts} from Proposition~\ref{prop:cg_subsets_give_necessary_conditions} are clearly satisfied.
We compute $\dim_\CC \Hom_S(I+S_{\mathcal{A}}, S/(I+S_\mathcal{A}))_{(0,0)} = 2$. Therefore, $[I]\notin \Slip_{2, X}$ by Proposition~\ref{prop:cg_subsets_give_necessary_conditions}.
\end{example}

Since the case of $X=\PP^{n_1}\times \cdots \times \PP^{n_d}$ is of greatest importance we present a version of Proposition~\ref{prop:cg_subsets_give_necessary_conditions}
in that case with $\mathcal{B} = \NN^d$.

\begin{proposition}\label{prop:ts_products_general}
Let $X=\PP^{n_1}\times\cdots\times\PP^{n_d}$ and fix a positive integer $r$. Let $\mathcal{C} = \{\mathbf{u}\in \ZZ^d_{\geq 0} \mid \dim_{\CC}S_{\mathbf{u}} \geq  r\}$. Assume that $\mathcal{A}\subseteq \ZZ_{\geq 0}^{d}$ is a subset satisfying the conditions:
\begin{enumerate}[label=(\roman*)]
    \item If $\mathbf{u}\in \mathcal{A}$, then $\mathbf{u}+\mathbf{e}_i\in \mathcal{A}$  where $\mathbf{e}_i$ is the $i$-th element of the standard basis of $\ZZ^d$.
    \item For every $\mathbf{u}\in \mathcal{C}$ there exist $k\in \ZZ_{>0}$, $\mathbf{v}\in \mathcal{C}\setminus \mathcal{A}$ and $\mathbf{w}, \mathbf{w}'\in \ZZ_{\geq 0}^d$ such that:
    \begin{itemize}
        \item $\mathbf{v}+\mathbf{w} \in \mathcal{C}\setminus \mathcal{A}$
        \item $\mathbf{v}+k\mathbf{w} = \mathbf{u}+\mathbf{w'}.$
    \end{itemize}
\end{enumerate}
Let $\mathfrak{a} = (S_{\mathbf{u}} \mid \mathbf{u}\in \mathcal{A})$.
If $[I]\in \Slip_{r, X}$, then $\dim_\CC \Hom_{S}(I+\mathfrak{a}, S/(I+\mathfrak{a}))_{\mathbf{0}} \geq r\cdot \dim X.$
\end{proposition}
\begin{proof}
    By Proposition~\ref{prop:cg_subsets_give_necessary_conditions} it is enough to verify that $\NN^d\setminus \mathcal{A}$ is $(r,X)$-sufficient. This follows from Lemma~\ref{lem:ifthensufficient}.
\end{proof}

Even when $X$ is the product of projective spaces and $\mathcal{B} = \NN^d$ there is still some freedom for the choice of $\mathcal{A}$. However, in what follows we restrict to $\mathcal{A}$ being the set of all degrees where the square of the irrelevant ideal of a fixed factor is non-zero.

\begin{corollary}\label{cor:ts_products}
Let $X=\PP^{n_1}\times \cdots \times \PP^{n_d}$ for some $d\geq 2$ and $n_1,\ldots, n_d \geq 1$.
Let $S[\PP^{n_i}] = \CC[\alpha_{i0}, \ldots, \alpha_{in_i}]$ be the Cox ring of the $i$-th factor.
Let $\mathfrak{a}_i$ be the extension of the irrelevant ideal $(\alpha_{i0}, \ldots, \alpha_{in_i})$ of $\PP^{n_i}$ to the Cox ring $S$ of $X$.
If $[I] \in \Slip_{r,X}$, then for every $i\in \{1,2,\ldots, d\}$
\[
\dim_\CC \Hom_S(I+\mathfrak{a}_i^2, S/(I+\mathfrak{a}_i^2))_\mathbf{0} \geq r\cdot \dim X.
\]
\end{corollary}
\begin{proof}
We want to apply Proposition~\ref{prop:cg_subsets_give_necessary_conditions} with $\mathcal{A} = \{(u_1, \ldots, u_d)\in \NN^d \mid u_i \geq 2\}$ and $\mathcal{B} = \NN^d$.
Conditions~\ref{con1:prop_ts}~and~\ref{con2:prop_ts} hold. We verify using Lemma~\ref{lem:ifthensufficient} that $\mathcal{B}\setminus \mathcal{A}$ is $(r,X)$-sufficient.

Let $[D] = (a_1, \ldots, a_d) \in \mathbb{N}^d$. Let $\mathbf{e}_i$ be the $i$-th coordinate vector of $\ZZ^d$ in the standard basis.
In the notation from Lemma~\ref{lem:ifthensufficient} we take $[E] = \sum_{j\neq i} \max\{a_j, r\}\mathbf{e}_j$, $[F] = \mathbf{e}_i$, $[G] = \sum_{j\neq i} \max\{r-a_j, 0\}\mathbf{e}_j$ and $k=a_i$.
\end{proof}

For the remainder of this section $X=\PP^m\times \PP^n$ for some $m,n \geq 1$. 
The Cox ring $S = S[X]$ of $X$ is of the form $\CC[\alpha_0, \ldots, \alpha_m, \beta_0 \ldots, \beta_n]$ where 
$\deg(\alpha_i) = (1,0)\in \mathbb{Z}^2$  for every $i\in \{0,1,\ldots, m\}$ and $\deg(\beta_j) = (0,1)\in \ZZ^2$ for every $j\in \{0,1,\ldots, n\}$. We show that $\Hilb_{S}^{h_{r,X}}$ is reducible for $r\in \{2,3\}$ using Corollary~\ref{cor:ts_products}. Let $\mathfrak{a} = (\alpha_0, \ldots, \alpha_m)$.

\begin{example}\label{exa:2pts}
We start with $r=2$ and $m=n=1$. Consider the ideal 
$I = (\beta_0\beta_1, \beta_0\alpha_0, \beta_0\alpha_1, \alpha_0^2).$
A direct calculation shows that $\dim_\mathbb{C} \Hom_S(I+\mathfrak{a}^2, S/(I+\mathfrak{a}^2))_{(0,0)} = 2 < 4 = \dim \Slip_{2,X}$.
Therefore, $[I] \notin \Slip_{2,X}$ by Corollary~\ref{cor:ts_products}.

If $m \geq 2$ or $n\geq 2$ take $I' = (\alpha_2, \ldots, \alpha_m, \beta_2, \ldots, \beta_n)+I$. We get $\dim_\CC \Hom_S(I'+\mathfrak{a}^2, S/(I'+\mathfrak{a}^2))_{(0,0)} \leq  2(m-1)+2(n-1) + 2 =2(m+n)-2 < 2(m+n)=\dim \Slip_{2,X}$. Therefore, $[I'] \notin \Slip_{2,X}$ by Corollary~\ref{cor:ts_products}.
\end{example}

\begin{example}\label{exa:3pts}
Similarly as in Example~\ref{exa:2pts} in order to show that $\Hilb_S^{h_{3,X}}$ is reducible it is enough to consider the cases $(m,n) \in \{(1,1), (2,1), (2,2)\}$.
\begin{enumerate}
\item If $(m,n)=(1,1)$ take $I=(\beta_0^2\beta_1, \alpha_0\beta_0, \alpha_0^3, \alpha_1^2\beta_0, \alpha_1\beta_0^2)$.
\item If $(m,n) = (2,1)$ take $I = (\beta_0\beta_1^2, \alpha_0\beta_0,\alpha_1\beta_0, \alpha_2\beta_0, \alpha_0^2, \alpha_0\alpha_1, \alpha_1^2)$.
\item If $(m,n) = (2,2)$ take $I=(\beta_1\beta_2^2, \beta_0^2, \beta_0\beta_1, \beta_0\beta_2, \alpha_0\beta_0, \alpha_0\beta_1, \alpha_1\beta_0, \alpha_1\beta_1, \alpha_2\beta_0, \alpha_2\beta_1, \alpha_0^2, \alpha_0\alpha_1, \alpha_1^2)$.
\end{enumerate}

In all cases $\dim_\CC \Hom_S(I+\mathfrak{a}^2, S/(I+\mathfrak{a}^2))_{(0,0)} < 3(m+n)$ so by Corollary~\ref{cor:ts_products}, those ideals are not in $\Slip_{3,X}$.
\end{example}

\section{Reducible multigraded Hilbert schemes of ideals of points in general position in a product of projective spaces}
\begin{lemma}\label{lem:1point}
If $X=\PP^{n_1}\times \cdots \times \PP^{n_d}$ for some positive integers $d,n_1, \ldots, n_d$, then $\Hilb^{h_{1,X}}_{S[X]}$ is irreducible.
\end{lemma}
\begin{proof}
It is enough to show that every homogeneous ideal $I$ of $S[X]$ with Hilbert function of $S[X]/I$ equal to $h_{1,X}$ is radical and $B(X)$-saturated. Let $J = (I\colon B(X)^\infty)$ and $K = \operatorname{rad}(I)$. If $J\neq I$, then there exists $\mathbf{u}\in \mathbb{N}^d$ with $J_{\mathbf{u}} = S[X]_{\mathbf{u}}$. It follows that for every $\mathbf{v} \in \mathbb{N}^d$ we have $J_{\mathbf{u}+\mathbf{v}} = S[X]_{\mathbf{u}+\mathbf{v}}$. This contradicts the fact that $S[X]/J$ has Hilbert polynomial $1$.

Assume that $K\neq I$. There is $\mathbf{u}\in \mathbb{N}^d$ with $K_{\mathbf{u}} = S[X]_{\mathbf{u}}$. Let $s = \max \{u_1, \ldots, u_d\}$.
Then, for every $\mathbf{v}\in \NN_{>0}^d$ we have $s\mathbf{v}-\mathbf{u} \in \NN^d$. It follows that $K_{\mathbf{v}} = S[X]_{\mathbf{v}}$  for every $\mathbf{v}\in \NN_{>0}^d$. In particular, $(I\colon B(X)^\infty) \neq I$.
\end{proof}

We summarize the results from earlier sections in the following result.

\begin{theorem}\label{thm:classification}
Let $r, d, n_1, \ldots, n_d$ be positive integers and $X=\PP^{n_1}\times \cdots \times \PP^{n_d}$.
The scheme $\Hilb^{h_{r, X}}_{S[X]}$ is irreducible if and only if one of the following holds:
\begin{enumerate}
\item $r=1$;
\item $d=1$ and $n_1 = 1$;
\item $d=1$ and $r\leq 3$.
\end{enumerate}
\end{theorem}
\begin{proof}
The case $r=1$ follows from Lemma~\ref{lem:1point}.

If $d=1$ and either $n_1=1$ or $r\leq 3$, then $\Hilb_{S[X]}^{h_{r, X}}$ is irreducible by Proposition~\ref{prop:classification_of_reducible_MGHS_pn}.
Therefore, cases 1.--3. correspond indeed to irreducible multigraded Hilbert scheme. We show that in all the other cases the scheme $\Hilb_{S[X]}^{h_{r, X}}$ is reducible.

We may and do assume that $n_1 \geq n_2 \geq \cdots \geq n_d$. Assume that $r\geq 4$ and consider two cases $n_1 \geq 2$ and $n_1=1$.
In the former case consider the natural map $\Hilb^{h_{r, X}}_{S[X]} \to \Hilb^{h_{r, \PP^{n_1}}}_{S[\PP^{n_1}]}$ described in  Corollary~\ref{cor:projection_criterion}.
It is surjective by Proposition~\ref{prop:projections_from_product_are_surjective}. The scheme $\Hilb_{S[\PP^{n_1}]}^{h_{r, \PP^{n_1}}}$ is reducible by Proposition~\ref{prop:classification_of_reducible_MGHS_pn}, therefore so is $\Hilb_{S[X]}^{h_{r,X}}$.
 In the second case we may assume that $d\geq 2$ and look at the natural surjective morphism
 \[
 \Hilb_{S[X]}^{h_{r, X}} \to \Hilb_{S[\PP^{n_1}\times \PP^{n_2}]}^{h_{r, \PP^{n_1}\times \PP^{n_2}}} = \Hilb^{h_{r, \PP^1\times \PP^1}}_{S[\PP^1\times \PP^1]}.
 \]
 The scheme $\Hilb_{S[\PP^1\times \PP^1]}^{h_{r, \PP^1\times \PP^1}}$ is reducible by Example~\ref{exa:p1_p1_is_reducible}. Thus, so is $\Hilb_{S[X]}^{h_{r, X}}$. 

We are left with the cases $d \geq 2$ and  $r\in \{2,3\}$. Again, using Proposition~\ref{prop:projections_from_product_are_surjective} it is enough to consider these cases with $d=2$. These were considered in  Examples~\ref{exa:2pts}~and~\ref{exa:3pts}.
\end{proof}

\appendix
\section{Results used in the proof of Theorem~\ref{thm:toric_fibration}}\label{app:1}
\subsection{Notation and background results on toric varieties}\label{sub:toric_notation}
We recall the notation and results from \cite{CLS11} that we use. Assume that $Z$ is a smooth complex $n$-dimensional projective toric variety. Let $\mathbb{T}\cong (\CC^*)^n$. We define $M_Z$ (or $M$) to be $\Hom(\mathbb{T}, \mathbb{C}^*)$---the lattice of characters of the torus $\mathbb{T}$. We denote the dual lattice with $N_Z$ (or $N$). We have a natural pairing $\langle -, - \rangle \colon M\times N \to \mathbb{Z}$ and its extension to the $\mathbb{R}$-vector spaces $M_{\mathbb{R}} = M\otimes_\ZZ \mathbb{R}$ and $N_\mathbb{R} = N \otimes_\ZZ \mathbb{R}$. Given a strongly convex rational polyhedral cone $\sigma$ in $N_{\RR}$ by $\sigma^\vee$ we denote the dual cone in $M_{\mathbb{R}}$ and by $U_{\sigma}$ we denote the spectrum of the group algebra $\CC[\sigma^\vee \cap M]$.
Every normal toric variety whose torus has character lattice $M$ is obtained by gluing a family of $U_{\sigma}$ where $\sigma$ belong to a fan $\Sigma$ in $N_{\mathbb{R}}$.

Let $\Sigma_Z=\Sigma$ be a fan of $Z$ and for a non-negative integer $d$ let $\Sigma(d)$ be the set of all  its $d$-dimensional cones. The Cox ring of $Z$ is then $S[Z] = \CC[\alpha_\rho \mid \rho \in \Sigma(1)]$. Furthermore, given a cone $\sigma\in \Sigma$ by $\sigma(1)$ we denote the set of all its $1$-dimensional faces. For $\sigma \in \Sigma$ let $\alpha^{\widehat{\sigma}} = \prod_{\rho\in\Sigma(1)\setminus \sigma(1)}\alpha_{\rho}$. Then $B(Z) = (\alpha^{\widehat{\sigma}}\mid \sigma\in \Sigma(n))$.

For every $\sigma \in \Sigma$ there is an isomorphism:
\begin{equation}\label{eq:isomorphism_toric_gluing}
\CC[\sigma^\vee \cap M] \to \left(S[Z]_{\alpha^{\widehat{\sigma}}}\right)_\mathbf{0}.
\end{equation}
If $m\in \sigma^\vee \cap M$ the corresponding element of $\CC[\sigma^\vee\cap M]$ is denoted by $\chi^m$.
The isomorphism is defined by $\chi^m\mapsto \prod_{\rho\in \Sigma(1)} \alpha_{\rho}^{\langle m, \mathbf{u}_\rho \rangle}$, where $\mathbf{u}_{\rho}\in \rho\cap N$ is the ray generator of $\rho$.

Recall from \cite[Sec.~3]{Cox95} that for every $\Pic(Z)$-graded $S[Z]$-module $M$ there is a corresponding quasicoherent sheaf $\widetilde{M}$ on $Z$ and that this correspondence defines an exact and essentially surjective functor from the category of $\Pic(Z)$-graded $S[Z]$-modules to the category of quasicoherent sheaves on $Z$. 
As in the case of $\PP^n$ we associate with a coherent sheaf $\mathscr{F}$ on $Z$ the $\Pic(Z)$-graded $S[Z]$-module $\Gamma_*(\mathscr{F}) = \bigoplus_{[D]\in\Pic(Z)} \Gamma(Z, \mathscr{F}(D))$. We have $\widetilde{\Gamma_*(\mathscr{F})} \cong \mathscr{F}$.
Furthermore, given a closed subscheme $R$ of $Z$, by \cite[Cor.~3.8]{Cox95} there exists a unique $B(Z)$-saturated ideal $I$ of $S[Z]$ such that $\widetilde{I}$ defines this subscheme. As in the case of $\PP^n$ we have the following result.
\begin{lemma}\label{lem:unique_saturated_ideal}
Let $R$ be a closed subscheme of $Z$ with ideal sheaf $\mathscr{I}$. The ideal $\Gamma_*(\mathscr{I})$ is $B(Z)$-saturated, so it is the unique $B(Z)$-saturated ideal of $S[Z]$ defining $R$.
\end{lemma}
\begin{proof}
Let $I$ be any ideal of $S[Z]$ with $\widetilde{I} \cong \mathscr{I}$. Let $s\in \Gamma(Z, \mathcal{O}_Z(D))\subseteq S[Z]$ be such that $\alpha^{\widehat{\sigma}} s \in \Gamma(Z, \mathscr{I}(D+\deg(\alpha^{\widehat{\sigma}}))\subseteq \Gamma_*(\mathscr{I})$ for every $\sigma \in \Sigma(n)$. 
If we restrict it to the open subset $U_{\sigma} \cong \Spec (\CC[\sigma^\vee \cap M])$ then we 
obtain an element $s|_{U_\sigma}\in (S[Z]_{\alpha^{\widehat{\sigma}}})_{[D]}$ such that
 $\alpha^{\widehat{\sigma}} s|_{U_\sigma} \in (I_{\alpha^{\widehat{\sigma}}})_{[D] + \deg (\alpha^{\widehat{\sigma}})}$.
 Multiplying it by $\alpha^{-\widehat{\sigma}}$ we conclude that $s|_{U_\sigma}\in (I_{\alpha^{\widehat{\sigma}}})_{[D]}$.
Since the sets $U_{\sigma}$ with $\sigma \in \Sigma(n)$ cover $Z$ we conclude by the sheaf property of $\mathscr{I}(D)$ that $s\in \Gamma(Z, \mathscr{I}(D)) \subseteq \Gamma_*(\mathscr{I})$.

This shows that $\Gamma_*(\mathscr{I})$ is a $B(Z)$-saturated ideal. The facts that $\widetilde{\Gamma_*(\mathscr{I})} \cong \mathscr{I}$ and that it is the unique ideal with these properties follow, as stated before the lemma, from the results in \cite{Cox95}.
\end{proof}

We consider only smooth and projective toric varieties. If we relax these assumptions, situation is more complicated. See \cite[Sec.~2.1]{Gal23} for a discussion of the properties of the ideal defining a subscheme of a singular or non-projective toric variety.

\subsection{Constructing a morphism to a multigraded Hilbert scheme}

Recall the notation from the proof of part~\ref{it:thm_fib_2} of Theorem~\ref{thm:toric_fibration}.

\begin{proposition}\label{prop:existance_of_morphism}
There is a morphism $\psi_{r,Z}\colon Z^r_{gen} \to \Hilb_{S[Z]}^{h_{r,Z}}$ given on closed points by $(p_1, \ldots, p_r) \mapsto [I(\{p_1, \ldots, p_r\})]$.
\end{proposition}
\begin{proof}
The subset $Z^r_{gen} \subseteq Z^r$ is open by \cite[Thm.~1.4]{BB21}. In particular, it has a natural scheme structure. Let $\mathscr{U} \subseteq Z^r_{gen}\times Z$ be the reduced closed subscheme $\coprod_{i=1}^r Z_i$ where $Z_i = \{\left((p_1, \ldots, p_r), q\right) \mid p_i=q\}$. The family $\mathscr{U}$ is flat over  $Z^r_{gen}$ since each $Z_i$ is mapped isomorphically to $Z^r_{gen}$. By construction the fiber over a closed point $(p_1,\ldots, p_r)$ of $Z^r_{gen}$ is the reduced subscheme $\{p_1, \ldots, p_r\}$ of $Z$. Let $\pi\colon Z^r_{gen}\times Z \to Z^r_{gen}$ be the projection.

Consider the exact sequence of $\mathcal{O}_{Z^r_{gen}\times Z}$-modules
\begin{equation}\label{eq:ses}
0\to \bigoplus_{[D]\in \Pic(Z)} \mathcal{I}_{\mathscr{U}}(D) \to \bigoplus_{[D]\in \Pic(Z)}\mathcal{O}_{Z^r_{gen}\times Z}(D) \xrightarrow{\eta} \bigoplus_{[D]\in \Pic(Z)} \mathcal{O}_{\mathscr{U}}(D) \to 0.
\end{equation}
Let $\mathcal{A} = \operatorname{im}(\pi_*\eta)$. We verify the following claims:
\begin{enumerate}[label=(\alph*)]
\item \label{it:app_1} we have $\pi_*(\bigoplus_{[D]\in \Pic(Z)}\mathcal{O}_{Z^r_{gen}\times Z}(D))\cong \mathcal{O}_{Z^r_{gen}}\otimes_\CC S[Z]$;
\item \label{it:app_2} $\mathcal{A}$ is a sheaf of $\mathcal{O}_{Z^r_{gen}}\otimes_{\CC} S[Z]$-algebras;
\item \label{it:app_3} $\mathcal{A}_{[D]}$ is a locally free sheaf of $\mathcal{O}_{Z^r_{gen}}$-modules of rank $h_{r,Z}([D])$ for every $[D]\in \Pic(Z)$.
\end{enumerate}
Claim~\ref{it:app_1} follows from \cite[Prop.~III.9.3]{Har77} since there is an isomorphism $\Gamma(Z, \mathcal{O}_Z(D))\cong S[Z]_{[D]}$. 
In the exact sequence \eqref{eq:ses} the $\mathcal{O}_{Z^r_{gen}\times Z}$-submodule $\bigoplus_{[D]\in \Pic(Z)} \mathcal{I}_{\mathscr{U}}(D)$ of the sheaf $\bigoplus_{[D]\in \Pic(Z)}\mathcal{O}_{Z^r_{gen}\times Z}(D)$ of $\mathcal{O}_{Z^r_{gen}\times Z}$-algebras  is a sheaf of ideals. By the left-exactness of the pushforward we get that the kernel of $\pi_*(\eta)$ is a sheaf of ideals of the sheaf $\mathcal{O}_{Z^r_{gen}}\otimes_\CC S[Z]$. Claim~\ref{it:app_2} follows.

Finally we address the third claim. By the definition of $\mathscr{U}$ and Lemma~\ref{lem:unique_saturated_ideal}, for every $z\in Z^r_{gen}$ we have 
$\dim_\CC H^0((Z^r_{gen} \times Z)_z, (\mathcal{I}_\mathscr{U}(D))_z) = \dim_\CC S[Z]_{[D]} - h_{r,Z}([D])$
 and $\dim_\CC H^0((Z^r_{gen} \times Z)_z, (\mathcal{O}_\mathscr{U}(D))_z) = r$. Moreover, both $\mathcal{I}_{\mathscr{U}}(D)$  and $\mathcal{O}_{\mathscr{U}}(D)$ are flat over $Z^r_{gen}$.
Therefore, by \cite[Cor.~III.12.9]{Har77} the sheaves of $\mathcal{O}_{Z^r_{gen}}$-modules $\pi_*(\mathcal{I}_{\mathscr{U}}(D))$ and $\pi_*(\mathcal{O}_{\mathscr{U}}(D))$ are locally free of rank $\dim_\CC S[Z]_{D}-h_{r,Z}([D])$ and $r$, respectively. 
In particular, if $h_{r,Z}([D]) < r$, then $\mathcal{A}_{[D]}\cong \mathcal{O}_{Z^r_{gen}}\otimes_\CC S[Z]_{[D]}$ is locally free of rank $h_{r,Z}([D])$.
Therefore, we may and do assume that $h_{r,Z}([D]) = r$. 
By the above, it is enough to show that $\pi_*(\eta)$ induces a surjection $\mathcal{O}_{Z^r_{gen}}\otimes_{\CC} S[Z]_{[D]} \to \pi_*(\mathcal{O}_{\mathscr{U}}(D))$. This can be checked on stalks over closed points, and by Nakayama's lemma it is enough to verify this on fibers. 
Let $z\in Z^r_{gen}$ correspond to the subscheme $R\subseteq Z$ and let $I_R$ denote its $B(Z)$-saturated ideal. Using \cite[Cor.~III.12.9]{Har77} it is enough to show that the natural map  $S[Z]_{[D]}\to H^0(Z, \mathcal{O}_R(D))$ is surjective. By Lemma~\ref{lem:unique_saturated_ideal} the kernel of this map is $(I_R)_{[D]}$ and we have
\[
r= \dim_\CC H^0(Z, \mathcal{O}_R(D)) = \dim_\CC S[Z]_{[D]}-\dim_\CC {(I_R)}_{[D]}.
\]
This finishes the proof of claim~\ref{it:app_3}.

Let $p\colon Z^r_{gen}\times \overline{Z} \to Z^r_{gen}$ be the natural projection where $\overline{Z} = \Spec S[Z]$. It follows from claims~\ref{it:app_1}~and~\ref{it:app_2} that there exists a family of closed subschemes of $Z^r_{gen}\times \overline{Z}$ over $Z^r_{gen}$ with structure sheaf $\mathcal{B}$ such that $p_*\mathcal{B} \cong \mathcal{A}$. Claim~\ref{it:app_3} ensures that this is an admissible family for the Hilbert function $h_{r,Z}$.
It follows from the universal property of $\Hilb_{S[Z]}^{h_{r,Z}}$ that there is a morphism $\psi_{r,Z}\colon Z^r_{gen} \to \Hilb_{S[Z]}^{h_{r,Z}}$ corresponding to this family.
By construction, on closed points, it maps $(p_1, \ldots, p_r)$ to $I(\{p_1, \ldots, p_r\})$.
\end{proof}

\subsection{Examples of applications of Theorem~\ref{thm:toric_fibration}}\label{subsec:examples}
We finish with two examples showing how the surjectivity of the map $\Slip_{r,X} \to \Slip_{r,Y}$ could be used to show that certain ideal is in $\Slip_{r,X}$. We first give a lemma which uses the notation from~\ref{sub:toric_notation}.

\begin{lemma}\label{lem:char_of_property_3} Let $f\colon X\to Y$ be a toric morphism between smooth projective toric varieties.
Let $S[X] = \CC[\alpha_\rho \mid \rho\in \Sigma_X(1)]$ and $S[Y] = \CC[\beta_\rho \mid \rho\in \Sigma_Y(1)]$
be the Cox rings of $X$ and $Y$, respectively. Let $\partial\colon N_X\to N_Y$ be the map corresponding to $f$.
Assume that we are given a homomorphism of rings $\varphi \colon S[Y]\to S[X]$ satisfying conditions
\ref{it:deflift1}~and~\ref{it:deflift2} from Definition~\ref{def:lift}. The homomorphism $\varphi$ is a lift of
$f$, if and only if we have 
\begin{equation}\label{eq:con_3}
\prod_{\rho\in \Sigma_Y(1)} (\varphi(\beta_{\rho}))^{\langle m, \mathbf{u}_\rho \rangle} = \prod_{\rho\in \Sigma_X(1)} \alpha_\rho^{\langle \delta(m), \mathbf{u}_{\rho} \rangle}
\end{equation}
for every every $m\in M_Y$, where $\delta \colon M_Y\to M_X$ is the dual map of $\partial \colon N_X\to N_Y$.
\end{lemma}
\begin{proof}
By $\widehat{f}\colon \widehat{X}\to \widehat{Y}$ we denote the morphism defined by $\varphi$ as in Definition~\ref{def:lift}. By \cite[Thm.~5.0.6]{CLS11} there is a unique morphism $f'\colon X\to Y$ such that $\pi_Y\circ \widehat{f} = f'\circ \pi_X$. We need to show that $f = f'$ if and only if \eqref{eq:con_3} holds for every $m\in M_Y$.
We have $f=f'$ if and only if they define the same morphism $U_\sigma \to U_{\sigma'}$ of affine toric varieties for every pair of cones $\sigma\in \Sigma_X$ and $\sigma'\in \Sigma_Y$ satisfying $\partial_\mathbb{R}(\sigma) \subseteq \sigma'$.

Recall that we have $\beta^{\widehat{\sigma'}} = \prod_{\rho\in \Sigma_Y(1)\setminus \sigma'(1)} \beta_\rho$ and $\alpha^{\widehat{\sigma}} = \prod_{\rho\in \Sigma_X(1)\setminus \sigma(1)} \alpha_\rho$. The map $U_\sigma \to U_{\sigma'}$ induced by $f$ corresponds to the homomorphism $\CC[(\sigma')^\vee\cap M_Y]\to \CC[\sigma^\vee\cap M_X]$ given by $\chi^m\mapsto \chi^{\delta(m)}$. On the other hand, the map $U_\sigma\to U_{\sigma'}$ induced by $f'$ corresponds to the map $\left(\varphi_{\beta^{\widehat{\sigma'}}}\right)_\mathbf{0}\colon \left(S[Y]_{\beta^{\widehat{\sigma'}}}\right)_\mathbf{0} \to \left(S[X]_{\alpha^{\widehat{\sigma}}}\right)_\mathbf{0}$. Therefore, $f$ and $f'$ induce the same map $U_\sigma\to U_{\sigma'}$ if and only if Equation~\eqref{eq:con_3} holds. Indeed, this is equivalent to the commutativity of the diagram
\begin{center}
\begin{tikzcd}[column sep=huge]
\CC[(\sigma')^\vee\cap M_Y] \arrow[r, "\chi^m\mapsto \chi^{\delta(m)}"] \arrow[d, "\chi^{m}\mapsto \prod_{\rho\in \Sigma_Y(1)} \beta_\rho^{\langle m, \mathbf{u}_\rho \rangle}"'] & \CC[\sigma^\vee\cap M_X]\arrow[d, "\chi^{m}\mapsto \prod_{\rho\in \Sigma_X(1)} \alpha_\rho^{\langle m, \mathbf{u}_\rho \rangle}"] \\
\left(S[Y]_{\beta^{\widehat{\sigma'}}}\right)_\mathbf{0} \arrow[r, "\left(\varphi_{\beta^{\widehat{\sigma'}}}\right)_\mathbf{0}" ] & \left(S[X]_{\alpha^{\widehat{\sigma}}}\right)_\mathbf{0}
\end{tikzcd}
\end{center}
where the vertical maps are the isomorphisms \eqref{eq:isomorphism_toric_gluing}.
\end{proof}

\begin{example}\label{exa:h1}
Let $X = \mathcal{H}_1$ be the Hirzebruch surface considered in Example~\ref{exa:hr}. Recall that we have $S[X] = \CC[\alpha_1, \alpha_2, \alpha_3, \alpha_4]$ with $\deg (\alpha_1) = (1,0)  = \deg (\alpha_3)$, $\deg (\alpha_2) = (1,1)$ and $\deg (\alpha_4)= (0,1)$. Let $Y = \PP^2$. Its fan is the complete fan in $\mathbb{R}^2$ with one-dimensional cones spanned by $v_0=(-1,1)$, $v_1=(1,0)$ and $v_2=(0,-1)$. Corresponding to these rays we have $S[Y] = \CC[\beta_0, \beta_1, \beta_2]$ with $\deg (\beta_0) = \deg (\beta_1) = \deg (\beta_2) = 1$.
Furthermore, the identity map of $\RR^2$ is compatible with the fans of $X$ and $Y$ so it gives a toric morphism $f\colon X\to Y$ that is the blowing up of $\PP^2$ at the torus invariant point $[0:0:1]$ \cite[pp.~132--133]{CLS11}.

We construct a lift of $f$ to a homomorphism $\varphi\colon S[Y] \to S[X]$. First we need to compute the pullback map $\phi\colon \Pic(Y)\to \Pic(X)$. We identify $\Pic(Y)$ with $\mathbb{Z}$ with basis $[D_0]=[D_1]=[D_2]$ and $\Pic(X)$ with $\mathbb{Z}^2$ with basis $([D_3], [D_4])$. Let $\psi \colon \RR^2\to \RR$ be the support function of the Cartier divisor on $Y$ corresponding to $v_2$. It follows from \cite[Prop.~6.2.7]{CLS11} that $\phi(1) = [D]$ where $D$ is the unique torus-invariant Cartier divisor on $X$ with support function $\psi$. This is the divisor corresponding to $u_2$. We obtain that $\phi(1) = (1,1)$.

Let $\varphi \colon S[Y]\to S[X]$ be defined by $\beta_0\mapsto \alpha_3\alpha_4$, $\beta_1\mapsto \alpha_1\alpha_4$ and $\beta_2\mapsto \alpha_2$. By the above computation of $\phi$ we see that it satisfies property~\ref{it:deflift1} from Definition~\ref{def:lift}.
Furthermore, $B(X) = (\alpha_1\alpha_2, \alpha_2\alpha_3, \alpha_3\alpha_4, \alpha_4\alpha_1)$ and $B(Y) = (\beta_0, \beta_1, \beta_2)$ so by Lemma~\ref{lem:characterization_of_property_2}  property~\ref{it:deflift2} from that definition is also fulfilled. In order to verify that property~\ref{it:deflift3} holds we use Lemma~\ref{lem:char_of_property_3}. Since $f$ is induced by the identity map on $\mathbb{R}^2$ Equation~\eqref{eq:con_3} takes the form of the following identity
\[
(\alpha_3\alpha_4)^{m_2-m_1}(\alpha_1\alpha_4)^{m_1}\alpha_2^{-m_2} = \alpha_1^{m_1}\alpha_2^{-m_2}\alpha_3^{-m_1+m_2}\alpha_4^{m_2}
\]
which holds for every $(m_1, m_2)\in\mathbb{Z}^2$.

Recall from Example~\ref{exa:hr} that $\Hilb^{h_{2, X}}_{S[X]}$ is not irreducible. Consider the ideal $I = (\beta_0^2, \beta_2) \subseteq S[Y]$. We have $[I]\in \Hilb^{h_{2, Y}}_{S[Y]} = \Slip_{2, Y}$ (see Proposition~\ref{prop:classification_of_reducible_MGHS_pn}).
It follows from Theorem~\ref{thm:toric_fibration} that there exists $[K]\in \Slip_{2,X} \subseteq \Hilb^{h_{2, X}}_{S[X]}$ with $\varphi^{-1}(K) = I$. However, $\varphi(I)\cdot S[X] = (\alpha_3^2\alpha_4^2, \alpha_2)$
implies that $K_{(2,2)} = (\alpha_3^2\alpha_4^2, \alpha_2)_{(2,2)}$. Consequently we get $K_{(2,0)} = \langle \alpha_3^2 \rangle$. 
Since $S[X]/(\alpha_3^2, \alpha_2)$ has Hilbert function $h_{2, X}$, it follows that $[(\alpha_3^2, \alpha_2)]$ is the unique closed point in the fiber of $\pi\colon \Hilb_{S[X]}^{h_{2, X}} \to \Hilb_{S[Y]}^{h_{2,Y}}$ over $[I]$ and is therefore a point of $\Slip_{2, X}.$
\end{example}

\begin{example}\label{ex:h1c}
We use the notation from Example~\ref{exa:h1}. Let $I = (\beta_1, \beta_2^2)\subseteq S[Y]$. We have $[I]\in \Hilb^{h_{2,Y}}_{S[Y]} = \Slip_{2,Y}$. We show that the fiber of $\pi$ over $[I]$ is one-dimensional and we identify the unique point in this fiber that belongs to $\Slip_{2,X}$.
Observe that $\varphi(I)\cdot S[X] = (\alpha_1\alpha_4,\alpha_2^2)$. Let $J = ((\alpha_1\alpha_4,\alpha_2^2)\colon B(X)^\infty)$. We have $J=(\alpha_1, \alpha_2^2)$ and 
\[
H_{S[X]/J}(a,b) = \begin{cases}
h_{2,X}(a,b) & \text{ if } b\geq 1\\
1 & \text{ if } b=0.
\end{cases}
\]
Let $[K]\in \Hilb_{S[X]}^{h_{2,X}}$ be a point in the fiber of $\pi$ over $[I]$. We claim that $(K\colon B(X)^{\infty}) = J$. By the definition of $J$ we have $(K\colon B(X)^{\infty})\supseteq J$. 
Suppose that the inclusion is strict. We may take an initial ideal of $(K\colon B(X)^{\infty})$ and then its saturation with respect to $B(X)$ to obtain a $B(X)$-saturated monomial ideal $J'$ with Hilbert polynomial of $S[X]/J'$ equal to $2$ that strictly contains $J$. Let $M$ be a monomial in $J'\setminus J$. If $M$ is of the form $\alpha_3M'$ or $\alpha_4M'$ for some monomial $M'$ then from the fact that $J'$ is $B(X)$-saturated and contains $(\alpha_1, \alpha_2^2)$ we conclude that $M'\in J'$. Therefore, we may and do assume that $M'$ is a monomial in $\alpha_1$ and $\alpha_2$. It follows that it is $1$ or $\alpha_2$ which contradicts the assumption on the Hilbert polynomial.

It follows from the above claims about the Hilbert function of $S[X]/J$ and the saturation of $K$ that the closed points of the fiber of $\pi$ over $[I]$ correspond to all the ideals from the set 
\[
\{J_{[s:t]}= (\alpha_2^2, \alpha_2\alpha_1, \alpha_1\alpha_4, \alpha_1(s\alpha_1+t\alpha_3) \mid [s:t]\in \PP^1\}.
\]
We claim that $[J_{[s:t]}]\in \Slip_{2, X}$ if and only if $[s:t] = [1:0]$. Observe that all $J_{[s:t]}$ with $st\neq 0$ are in one $\operatorname{GL}_4(\CC)$-orbit, so if any of them is in $\Slip_{2, X}$ then all $[J_{[s\colon t]}]$ are in $\Slip_{2,X}$. However, Example~\ref{exa:hr} shows that $[J_{[0:1]}]\notin \Slip_{2,X}$. Therefore, $[J_{[s:t]}]$ with $t\neq 0$ are not in $\Slip_{2,X}$.
On the other hand, by Theorem~\ref{thm:toric_fibration} there is at least one $[s:t]\in \PP^1$ such that $[J_{[s:t]}] \in \Slip_{2,X}$. Hence $[J_{[1:0]}] \in \Slip_{2,X}$.
\end{example}

\end{document}